\documentclass[a4paper,reqno,10pt]{amsart}

\usepackage[a4paper,left=3cm,right=3cm,top=3cm,bottom=3.5cm]{geometry}

\usepackage[utf8]{inputenc}
\usepackage[T1]{fontenc}
\usepackage[english]{babel}

\usepackage{mathrsfs}
\usepackage{stmaryrd}

\usepackage{amsmath,amssymb,amsfonts,amsthm}
\usepackage{mathtools}
\usepackage{mathrsfs}
\usepackage{stmaryrd}

\usepackage{booktabs}
\usepackage{float}
\usepackage{xcolor}
\usepackage{marginnote}
\usepackage{dsfont}
\usepackage{comment}
\usepackage{lastpage}

\usepackage[hidelinks,unicode=true]{hyperref}

\renewcommand{\paragraph}[1]{$\empty$\\\textit{#1}.\hspace{0.5em}}

\numberwithin{equation}{section}

\usepackage{enumitem}

{\begin{itemize}[
	label=$-$, %changer puces
	leftmargin=*, %decalage horizontal par rapport au reste du texte
]}{\end{itemize}}

\newtheorem{thm}{Theorem}[section]

\newtheorem{lem}[thm]{Lemma}

\newcommand{\fa}{\forall \,}

\renewcommand{\and}{\qquad\mathrm{and}\qquad}

\renewcommand{\leq}{\,\leqslant\,}
\renewcommand{\geq}{\,\geqslant\,}
\newcommand{\<}{\,<\,}
\renewcommand{\>}{\,>\,}
\renewcommand{\=}{\,=\,}
\newcommand{\Llra}{\Longleftrightarrow}

\newcommand{\lra}{\longrightarrow}

\newcommand{\xra}[2]{\;\;\xrightarrow[#2]{#1}\;\;}
\newcommand{\lms}{\longmapsto}

\newcommand{\bracket}[1]{\llbracket #1 \rrbracket}

\newcommand{\ps}[1]{\left\langle #1 \right\rangle}
\newcommand{\dps}{\displaystyle}
\newcommand{\event}[1]{\left\{#1\right\}}
\newcommand{\vabs}[1]{\left|#1\right|}
\newcommand{\parent}[1]{\left(#1\right)}
\newcommand{\intervalle}[1]{\left[#1\right]}
\newcommand{\norme}[2]{\left\|#1\right\|_{#2}}

\newcommand{\dist}{\dist}

\newcommand{\mtext}[1]{\mbox{\rm #1}}
\newcommand{\textm}[1]{\;\; \mbox{\rm #1}}
\newcommand{\deff}[5]{\begin{tabular}{lccl} $#1$ & $#2$ & $\lra$ & $#4$\\  & $#3$ & $\lms$ & $#5$ \end{tabular}}
\newenvironment{acc}{\left\{\begin{tabular}{ll}}{\end{tabular}\right.}

\renewcommand{\d}{\mathrm{d}}

\newcommand{\Hess}{\mathrm{Hess}}

\newcommand{\C}{\mathrm{C}}
\newcommand{\D}{\mathrm{D}}
\newcommand{\E}{\mathrm{E}}

\newcommand{\G}{\mathrm{G}}
\renewcommand{\H}{\mathrm{H}}
\newcommand{\I}{\mathrm{I}}

\renewcommand{\L}{\mathrm{L}}
\newcommand{\M}{\mathrm{M}}

\renewcommand{\P}{\mathrm{P}}
\newcommand{\Q}{\mathrm{Q}}

\renewcommand{\S}{\mathrm{S}}
\newcommand{\T}{\mathrm{T}}

\newcommand{\V}{\mathrm{V}}

\newcommand{\1}{\mathds{1}}

\newcommand{\EE}{\mathbb{E}}

\newcommand{\NN}{\mathbb{N}}

\newcommand{\RR}{\mathbb{R}}

\newcommand{\CCC}{\mathcal{C}}

\newcommand{\MMM}{\mathcal{M}}

\newcommand{\PPP}{\mathcal{P}}

\newcommand{\tmix}{t_{\text{mix}}}

\newcommand{\Law}{\mtext{Law}}
\renewcommand{\dist}{\mtext{dist}}

\newcommand{\Kullback}{\mtext{Kullback}}

\newcommand{\Wasserstein}{\mtext{Wasserstein}}

\newcommand{\Lip}{\mtext{Lip}}
\newcommand{\Ric}{\mtext{Ric}}
\newcommand{\Var}{\mtext{Var}}

\title{On the cutoff phenomenon for Dyson--Laguerre processes}
\author{Samuel Chan-ashing}
\address{
DMA, École normale supérieure, Université PSL, CNRS, 75005 Paris, France \newline
CEREMADE, Université Paris-Dauphine, PSL, CNRS \newline
CMAP, Inria, CNRS, École polytechnique, Institut Polytechnique de Paris, 91120 Palaiseau, France}
\email{samuel.chan-ashing@ens.psl.eu}

\date{\today}



\begin{document}

\begin{abstract}
We study the convergence to equilibrium in high dimensions, focusing on explicit bounds on mixing times and the emergence of the cutoff phenomenon for Dyson--Laguerre processes. These are interacting particle systems with non-constant diffusion coefficients, arising naturally in the context of sample covariance matrices. The infinitesimal generator of the process admits generalized Laguerre orthogonal polynomials as eigenfunctions.

Our analysis relies on several distances and divergences, including an intrinsic Wasserstein distance adapted to the non-Euclidean geometry of the process. Within this framework, we employ tools from Riemannian geometry and functional inequalities. In particular, we establish exponential decay and derive a regularization inequality for the intrinsic Wasserstein distance via comparison with relative entropy.
\end{abstract}

\maketitle

\setcounter{tocdepth}{1}
\tableofcontents

%RAJOUTER LES EQUATIONS \begin{equation} \end{equation} and \begin{equation}\label{eq:ref} \begin{aligned}

\section{Introduction and main results}

The \emph{cutoff phenomenon} describes an abrupt transition in the convergence to equilibrium of a parameter-indexed family of stochastic processes for instance, with the dimension as the parameter.
In essence, it reflects a competition between intrinsic relaxation toward equilibrium and initial-data dependence; the parameter determines the regime in which this competition is resolved. When cutoff occurs, the distance to equilibrium remains near its maximum until a characteristic time and then falls to near zero.
This notion was introduced by David Aldous and Persi Diaconis in the 1980s in their study of random walks on finite sets; see for instance \cite{AD86, Dia96}. In the 1990s, Laurent Saloff-Coste extended this framework to Markov diffusion processes, in connection with Nash-Sobolev-type functional inequalities and heat kernel estimates \cite{SC94, SC04}.

A key preliminary concept in characterizing the cutoff phenomenon for general Markov processes is the \emph{product condition}, introduced by Yuval Peres. Guan-Yu Chen and Laurent Saloff-Coste showed in \cite{CSC08} that this condition is sufficient in $\L^p$ distance for $p>1$. However, they also presented two counterexamples, in total variation distance (i.e., $p=1$), due to Aldous and Igor Pak, showing that the product condition is not sufficient for general discrete processes.
Their approach is based on functional inequalities, which provide control over both upper and lower bounds on the mixing time. In a similar vein, Justin Salez recently demonstrated in \cite{Sal21, Sal25} that under an additional curvature assumption, the product condition does imply cutoff for finite Markov chains and diffusions in total variation distance. The proof builds on a two-way bridge between total variation and entropy: Pinsker's inequality bounds total variation via relative entropy, while the converse involves the notion of \emph{varentropy}, controlled by Fisher information.

While these general methods do not provide sharp estimates of the mixing time, Jeanne Boursier, Djalil Chafaï and Cyril Labbé identified in \cite{BCL} the precise cutoff time for the Dyson--Ornstein--Uhlenbeck diffusion in $\RR^n$. In the non-interacting case ($\beta = 0$), the process reduces to the Ornstein--Uhlenbeck process, for which new convergence results are also obtained. Previous analyses of its equilibrium behavior include \cite{Lac05, BHP20a, BHP20b, BP20}.
These dynamics belong to the class of \textit{Dyson diffusions}, a family of interacting particle systems arising from random matrix theory. We pursue this line of work by studying integrable Dyson-type dynamics associated with the $\beta$-Laguerre ensembles. These dynamics arise from the eigenvalue evolution of the Wishart process, which can be seen as the analog of the Dyson--Ornstein--Uhlenbeck process for covariance matrices. Such processes were notably studied by Marie-France Bru and Ezéchiel Kahn; see \cite{Bru89, Bru91, JK20, Kah21}. In the absence of interaction ($\beta = 0$), they reduce to the Cox--Ingersoll--Ross process, also known as the squared Bessel process.

We explore the cutoff phenomenon with respect to the total variation distance, the Kullback--Leibler divergence, the \( \L^2 \) distance, and the intrinsic Wasserstein distance. In particular, we do not consider the Hellinger distance, as it can be compared to the total variation distance (see \cite{BCL}), nor the chi-squared (\(\chi^2\)) distance, which coincides with the \( \L^2 \) distance. Moreover, the case of \( \L^p \) distances with \( p > 1 \) can be handled by interpolation using the results obtained for the \( \L^2 \) distance. Finally, for the Fisher distance, a lower bound can be easily derived using curvature-based arguments.
These distances and divergences are related through several functional inequalities: the total variation distance and the Kullback--Leibler divergence are linked via the Pinsker inequality and its converse, as in \cite{Sal25}; and the Kullback--Leibler divergence can be bounded in terms of the $\L^2$ norm, as in \cite{BCL}.

A structural explanation of the cutoff phenomenon for certain overdamped Langevin diffusions with positive curvature in Euclidean space is given by Djalil Chafaï and Max Fathi in \cite{CF24}. Their results extend to a broad class of models with convex interactions, beyond the Gaussian and product settings.
In particular, they obtain a regularization estimate for the Kullback--Leibler divergence and the Fisher information through the Euclidean Wasserstein distance. We will show that this regularization is compatible with the intrinsic Wasserstein distance, which in turn provides a useful tool for estimating the mixing time.
%Note that we can obtain such results for Fisher if we could show that we have monotonicity for this divergence: which we have but for a non-Euclidean Fisher.

\subsection{The Dyson--Laguerre process}

In what follows, we study the \textit{Dyson--Laguerre process} (DL), defined as the solution \( X^n\) to the stochastic differential equation (SDE)
\begin{equation}\label{eq:DL}
\d X_{t}^{i,n} \= \sqrt{2  X_{t}^{i,n}} \, \d  B^i_t + \parent{\alpha_n -  X_{t}^{i,n} + \; \frac{\beta}{2} \sum_{j \neq i}\frac{ X_{t}^{i,n}+ X_{t}^{j,n}}{ X_{t}^{i,n}- X_{t}^{j,n}}}\d t,  \qquad  X_0 \= x_0^n \in \RR_+^n,
\end{equation}
where the parameters satisfy $n\geq 1$ and $\alpha_n,\beta\geq0$.

\paragraph{Existence and uniqueness}
The system may blow up, either due to degeneracies of the diffusion coefficients at the boundary (where they may vanish), or as a consequence of particle collisions. To avoid such pathologies, we focus on a parameter regime ensuring well-posedness of the dynamics.

Each coordinate evolves as a Cox--Ingersoll--Ross process (see \eqref{eq:CIR}), and they repel each other according to a Coulomb-type interaction. Under the ordering condition $0 \leq x_0^{1,n} \< \cdots \< x_0^{n,n}$, it is proved in \cite[Corollary 8]{GM13} that for all $\beta \geq 1$, the system \eqref{eq:DL} admits a unique global strong solution with no collisions occurring at any finite time.

Existence and collision questions for $\beta \geq 1$ are addressed in \cite[Theorem 2.2]{GM14}. Broader ranges of the parameter $\alpha_n - (n-1)\beta/2$ are treated in \cite[Theorem 2.4]{JK20}. The analysis is in the spirit of L. Chris G. Rogers and Zhan Shi \cite{RS93} for symmetric matrices, who investigated the $n \to \infty$ asymptotics of the empirical measure of particle systems whose interaction is the derivative of a potential.
More precisely, it is shown that \eqref{eq:DL} admits a strong, pathwise unique solution on $[0,+\infty)$ with no particle collisions whenever
\begin{equation}\label{eq:existence}
\beta \geq 1 \and \fa n \geq 1, \quad \alpha_n -(n-1)\frac{\beta}{2} \> 1.
\end{equation}
Although cutoff is expected under broader conditions, we will work throughout under the standing assumption \eqref{eq:existence}.

\paragraph{Matrix cases}
Let us focus here on the real case \( \beta = 1 \), where $\alpha_n$ is a positive even integer. Analogous results hold in the complex case with \( \beta = 2 \). Let $n,m\in\NN^\ast$ with $n\leq m$, and consider a \textit{rectangular Ornstein--Uhlenbeck process} $( M_t)_{t \geq 0}$, taking values in the space of $n\times m$ real matrices, with independent Ornstein--Uhlenbeck (OU) entries. This process satisfies the SDE
\begin{equation}\label{eq:rOU}
    \d  M_t \= \sqrt{\frac{m}{2}} \d  W_t - \frac{1}{2}  M_t \d t,
\end{equation}
where $ W$ is an $n\times m$ matrix-valued standard Brownian motion. The covariance matrix $ M_t^\dag  M_t$, where $\dag$ denotes the conjugate transpose, is known as the \textit{Wishart process}; see \cite{Bru91} for further details. This is the covariance-matrix version of Freeman J. Dyson's observation \cite{Dys62}.

\noindent Let us now consider, for each $t\ge0$, the eigenvalues of $ M_t^\dag  M_t$. A remarkable fact from \cite[Theorem 1]{Bru91} is that the spectrum evolves as a Dyson--Laguerre particle system \eqref{eq:DL} with parameters
\[\alpha_n \= 2 m \quad\mtext{and}\quad \beta\=1,\]
and remains, almost surely, in the convex domain of ordered nonnegative coordinates
\[X_t^n \in \D_n \;:=\; \event{x \in \RR^n :\; 0 \leq x_1 \< \cdots \< x_n}.\]
The eigenvalues thus evolve as ordered coordinates and form a Coulomb gas. As shown in \cite{JK20}, this confinement holds for general values of \( \beta \).

\paragraph{Invariant measure}
Since we are observing the eigenvalues of the ``square'' of a rectangular OU process, that is, a sum of Gaussian components, it is natural to expect that the invariant measure resembles a matrix analogue of a chi-squared distribution. This is confirmed in \cite[Section 3.2]{For10}, where Forrester shows that the invariant measure of the associated particle system corresponds to the $\beta$-Laguerre ensemble, in the cases of real symmetric ($\beta=1$), complex Hermitian ($\beta=2$), and quaternion self-dual ($\beta=4$) matrices. This ensemble naturally involves the Gamma distribution. More generally, according to \cite[Proposition 2.8]{JK20}, the DL process \eqref{eq:DL} admits a unique stationary probability measure $\pi_{\beta}^n$, whose density with respect to the Lebesgue measure is given, up to normalization, by
\begin{equation}\label{eq:gc}
\d \pi_{\beta}^{n}(x) \;:=\; \frac{\1_{{(x_1,\ldots,x_n)\in\overline{\D}_n}} }{C_n^\beta} \prod_{i=1}^n x_i^{\alpha_n-(n-1)\frac{\beta}{2}-1} e^{-x_i} \prod_{i > j}(x_i-x_j)^{\beta} \d x_i,\end{equation}
where $C_n^\beta$ is a normalizing constant.
The invariant measure is the Gibbs measure associated with the energy given for $x\in\D_n$ by
\[
\E(x) \;:=\; \sum_{i=1}^n \parent{x_i - \parent{\alpha_n - (n-1) \frac{\beta}{2} - 1} \log(x_i) - \beta \sum_{j \neq i} \log(x_i-x_j)}.
\]
On its natural domain, this energy is convex; consequently, the Coulomb gas $\pi_{\beta}^{n}$ is log-concave.

\subsection{Main results}
 
The mixing time is defined for any distance or divergence ($\dist$) by 
\[
\tmix^n(\varepsilon) \;:=\; \inf \event{t \geq 0 \;:\; \dist \parent{\Law(X_t^n) \mid \pi_{\beta}^{n}} \leq \varepsilon}.
\]
For probability measures $\mu$ and $\nu$ on the same space, their total variation distance is
\[
\norme{\mu-\nu}{\T\V} := \sup_A\vabs{\mu(A)-\nu(A)} \in [0,1],
\]
the $\L^p$ distance for $p>1$ is defined by
\[
\norme{\mu-\nu}{\L^p(\nu)} := \norme{\frac{\d \mu}{\d \nu} - 1}{p} \in [0,+\infty],
\]
(when $\nu$ is the invariant measure $\pi_\beta^n$, we denote, for brevity, $\norme{\mu-\nu}{p}$), and the relative entropy (Kullback--Leibler divergence) is defined by
\[
\Kullback(\mu \mid \nu) := \int \frac{\d \mu}{\d \nu} \log \parent{\frac{\d \mu}{\d \nu}} \d \nu \= \int \log \parent{\frac{\d \mu}{\d \nu}} \d \mu \in [0,+\infty],
\]
with the convention $\norme{\mu-\nu}{\L^p(\nu)} = \Kullback(\nu \mid \mu) = +\infty$ if $\nu$ is not absolutely continuous with respect to $\mu$.
The \textit{intrinsic Wasserstein distance} is obtained from Section \ref{sec:iwd} and Subsection \ref{sec:wgdl}, and is given by
\begin{equation}\label{eq:iwd}
\Wasserstein_{}(\mu, \nu) \;:=\; 2 \sqrt{\inf_{( X, Y)} \EE\intervalle{\sum_{i=1}^n \parent{ \sqrt{ X_i}-\sqrt{ Y_i}}^2}}.
\end{equation}
In what follows, we set
\[\max{} \;:=\; \begin{acc}
	$1$ & for $\dist \= \T\V$,\\
	$+\infty$ & for $\dist \in \{\L^2, \Kullback, \Wasserstein_{}\}$.
\end{acc}\]
Using Pinsker’s inequality and its reverse form (see e.g., \cite{Sal25}), we can switch between total variation and Kullback--Leibler divergence when deriving upper or lower bounds, choosing whichever is more convenient.

\paragraph{The real matrix case ($\beta=1$)} Let us make this observation in the case of real matrices, i.e., for \( \beta = 1 \), where the associated Gibbs measure is the $\beta$-Laguerre ensemble. Recall that, in this matrix setting, the Dyson--Laguerre process is obtained with $\alpha_{n,m} \= 2 m$. To derive an upper bound in this case, we project the DL process onto a rectangular OU process.

\begin{thm}[Cutoff for the real matrix case ($\beta=1$)]\label{thm:cmc}
Let ${( X^n_t)}_{t\geq0}$ be the Dyson--Laguerre process \eqref{eq:DL} started at $x_0^n$ and resulting from the rectangular Ornstein--Uhlenbeck process \eqref{eq:rOU}.
Then for $\dist \in \{ \T\V, \L^2, \Kullback, \Wasserstein_{}\}$ and for all $\varepsilon \in (0,1)$,
\[\lim_{n\to\infty} \dist(\Law( X^n_{t_n})\mid \pi_1^n) \=\begin{cases}
	\max & \text{if $t_n=(1-\varepsilon)c_n$},\\
	0 & \text{if $t_n=(1+\varepsilon)c_n$},
\end{cases}\]
assuming that $c_n \;:=\; \log\parent{\frac{1}{n} \sum_{i=1}^{n} x_{0}^{i,n}} \vee \log n \xra{}{n\to\infty} +\infty$.

For $\dist \= \Wasserstein_{}$, the same conclusion remains valid, but the critical time $c_n$ is such that $\log\parent{\frac{1}{n} \sum_{i=1}^{n} x_{0}^{i,n}} \vee \log n \leq c_n \leq \log\parent{\sum_{i=1}^{n} x_{0}^{i,n}} \vee \log n$.

In particular, when $\log\parent{\sum_{i=1}^{n} x_{0}^{i,n}} = O( \log n)$, we recover the result with a mixing time of order $c_n \= \log n$.
\end{thm}

\paragraph{Cutoff in the general case}
The upper bound in the general case is obtained using the curvature-dimension condition and a regularization argument. The crucial step in the proof is the regularization result stated in Lemma \ref{lem:reg}, which relies on the intrinsic framework developed in Appendix \ref{sec:iwd}.

\begin{thm}[Universal cutoff]\label{thm:CDL}
Let ${( X^n_t)}_{t\geq0}$ be the Dyson--Laguerre process \eqref{eq:DL} started at $x_0^n$. Then for $\dist \in \{ \T\V, \Kullback, \Wasserstein_{}\}$ and for all $\varepsilon \in (0,1)$, there exists $c_n$ such that
\[\lim_{n\to\infty} \dist(\Law( X^n_{t_n})\mid \pi_{\beta}^{n}) \=\begin{cases}
	\max & \text{if $t_n=(1-\varepsilon)c_n$},\\
	0 & \text{if $t_n=(1+\varepsilon)c_n$},
\end{cases}\]
and, assuming that the following lower bound tends to infinity as $n\to\infty$,
\[
\log \parent{\frac{1}{\sqrt{\alpha_n n}} \vabs{\sum_{i=1}^{n} x_{0}^{i,n}}} \vee \log \parent{\sqrt{\alpha_n n}} \leq c_n \leq \log\vabs{\sum_{i=1}^{n} x_{0}^{i,n}} \vee \log\parent{\alpha_n n}.
\]
In particular, if $\limsup_n \parent{\frac{1}{n} \sum_{i=1}^n x_0^{i,n} - \alpha_n} \leq 0$, then the mixing time is
\[
\frac{1}{2} \log \parent{\alpha_n n} \leq c_n \leq \log\parent{\alpha_n n}.
\]
\end{thm}

\paragraph{Organization of the paper}
The DL process admits multiple interpretations, and throughout this paper, we adopt two complementary viewpoints:

Euclidean viewpoint (Section \ref{sec:epov}). The Euclidean perspective arises from a square-root transformation, under which the process naturally becomes a Langevin diffusion (with a constant diffusion coefficient). This yields a Euclidean framework that enables the application of classical tools, but we emphasize that, from the standpoint of generality, the change of variables considered is not optimal. Certain structural properties, such as convexity, need to be preserved under this transformation, which restricts its utility in some analytic settings.
This strategy may not apply in more general settings, for instance, if the transformed Langevin process does not correspond to a log-concave gas.
Indeed, log-concavity is not generally preserved under a change of variables. However, in the present setting, transformation behaves well and this property is retained.

Riemannian viewpoint. The Riemannian perspective interprets the DL process as evolving on a curved space, where the non-constant diffusion coefficient plays the role of a metric. This geometric viewpoint opens directions for analyzing more general diffusions, particularly in cases where a change of variables does not suffice to reduce the problem to a Euclidean setting.
In this context, the universal lower bound stated in Lemma \ref{lem:LBDL} is proved in Section \ref{sec:pulb}, following the spectral approach initiated in \cite{SC94}. The cutoff phenomenon is thus established uniformly across all distances, as summarized in Subsection \ref{ss:krc}, by leveraging the collection of functional inequalities derived throughout the paper. The estimates for the mixing time in the matrix setting (Theorem \ref{thm:cmc}) and in the general setting (Theorem \ref{thm:CDL}) are proved in Sections \ref{sec:pubmc} and \ref{sec:pubgc}, respectively. In particular, the matrix case relies on precise estimates for the mixing time of an Ornstein--Uhlenbeck process with general coefficients, detailed in Appendix \ref{sec:mtou}. The intrinsic Wasserstein distance and its associated functional inequalities are defined and studied in Appendix \ref{sec:iwd}.

\paragraph{Further work}
Ongoing work investigates the cutoff phenomenon for the remaining integrable Dyson-type dynamics, notably those associated with the beta-Jacobi ensembles. This involves a detailed analysis of the curvature of the process and the construction of a suitable intrinsic Wasserstein distance. Further directions include studying the cutoff for the dynamics of $n$ i.i.d. ergodic particles and particle systems with weak interactions.

\section{Additional comments}

\subsection{Further analysis of the Dyson--Laguerre process}

Further details on the noninteracting process, its spectral analysis, and the associated carré du champ operator follow.

\paragraph{The non-interacting case}
The case \(\beta=0\) corresponds to product dynamics, in which particles evolve independently, each according to a one-dimensional Dyson--Laguerre process. In contrast to the Ornstein--Uhlenbeck setting, no simple control over equilibrium distances is available.

More precisely, each coordinate evolves as a Cox--Ingersoll--Ross process: the $i$-th particle satisfies the SDE
\begin{equation}\label{eq:CIR}
\d X_{t}^{i,n} \= \sqrt{2 X_{t}^{i,n}} \, \d  B_{t}^{i,n} + \parent{\alpha_n-X_{t}^{i,n}} \d t, \qquad  X_0^{i,n} \= x_0^{i,n} \in \RR_+,
\end{equation}
where \( \alpha_n \in \mathbb{R}^+ \) denotes a dimension parameter and \(( B^{i,n})_{1\leq i \leq n}\) are independent standard Brownian motions. While collisions between particles may occur, the absence of any repulsive term prevents singularities, ensuring that the dynamics remain well-posed.
As shown by Feller in \cite{Fel51} (see also \cite{dC16, LL97}), the law of $X_{t}^{i,n}$ admits an explicit density involving special functions, notably the modified Bessel function of the first kind. At a fixed time, the Cox--Ingersoll--Ross process follows a non-central chi-squared distribution with \( \alpha_n \) degrees of freedom. Its invariant law is the Gamma distribution $\Gamma\parent{\alpha_n,1}$.
As a result, the invariant measure $\pi_0^n$ of the full system is a product of independent, non-centered Gamma distributions.
This tensorized structure allows for full explicit control over both the dynamics and the invariant law. However, it proves of limited use for our purposes, compared to the Ornstein--Uhlenbeck case in \cite{BCL}.
Indeed, despite the availability of exact expressions, estimating distances between non-centered Gamma laws remains a subtle task. For this reason, we do not treat separately the interacting and non-interacting cases, and instead focus on the DL process as a whole, avoiding any projection.

\paragraph{Spectral analysis}
It is important to note that the generator has a remarkable structure. Indeed, in the non-interacting case the coefficients of the first- and second-order terms are affine functions of the coordinates. Thus one expects the eigenfunctions to be polynomials, and in particular orthogonal polynomials, since the generator is self-adjoint and its eigenfunctions can be chosen orthogonal to each other. This phenomenon was observed by Lassalle \cite{Lasc}.
Here we focus on the general $\beta$ case, which leads naturally to multivariate Laguerre polynomials. Access to explicit eigenfunctions is a significant advantage in the analysis of cutoff.

Let $\L_{\text{sym}}^2(\RR_+^n,\pi_{\beta}^{n})$ denote the Hilbert space of symmetric functions in \( n \) variables \( (x_1, \dots, x_n) \) on the positive quadrant \( \{x_i > 0, 1 \leq i \leq n\} \), square-integrable with respect to the invariant measure \( \pi_{\beta}^{n} \). The infinitesimal generator of the DL process acts on this space as a second-order differential operator
\begin{equation}\label{eq:IG}
\G_{\D\L}(f) \= \sum_{i=1}^{n} x_i \partial_{ii}^2 f + \alpha_n\sum_{i=1}^{n} \partial_i f - \sum_{i=1}^{n} x_i \partial_i f + \frac{\beta}{2} \sum_{i=1}^{n} \sum_{\substack{j=1 \\ j \neq i}}^{n} \frac{x_i+x_j}{x_i - x_j}\partial_i f.\end{equation}
The operator $\G_{\D\L}$ is self-adjoint on $\L_{\text{sym}}^2(\RR_+^n,\pi_{\beta}^{n})$ and preserves the space of symmetric polynomials. Its spectrum is the set of non-positive integers $-\NN$, the eigenspaces are finite-dimensional, and the eigenfunctions are the generalized Laguerre polynomials. The spectral gap is equal to $1$, and the eigenspace corresponding to the eigenvalue $-1$ is generated by an explicitly given eigenfunction, which is an affine function
\begin{equation}\label{eq:efsp}
\varphi_{n}(x) \;:=\; \phi_n(x) - \alpha_n n - \frac{\beta}{2}(n-1)^2 \qquad\mtext{where}\qquad \phi_n(x) \;:=\; x_1 + \cdots + x_n.
\end{equation}
Note that the process is referred to as the \textit{Dyson--Laguerre} process, while its infinitesimal generator is simply called the \textit{Laguerre operator}.

\paragraph{Carré du champ}
Since the diffusion coefficient is non-constant, the carré du champ is no longer the squared Euclidean gradient. Instead we obtain the deformed carré du champ
\begin{equation}\label{eq:cdc}
\Gamma f \= \sum_{i=1}^n x_i (\partial_{i} f)^2.
\end{equation}
Because $\Gamma$ appears systematically (e.g., in the derivation of relative entropy, the log–Sobolev inequality, etc.), Euclidean properties no longer align with those of the diffusion. In particular, the Euclidean geometry is not the right ambient structure for direct analysis of this process. We will therefore adopt a geometry adapted to $\Gamma$ that fits the dynamics in Section \ref{sec:wgdl}.

\paragraph{Alternative parametrization} In a broader setting, one may consider a DL process with arbitrary coefficients.
The process satisfies the following SDE
\[
\d X_{t}^{i,n} \= \sigma \sqrt{ X_{t}^{i,n}} \, \d  B^i_t + \parent{\alpha_n - \lambda  X_{t}^{i,n} + \beta \sum_{j \neq i}\frac{ X_{t}^{i,n}+ X_{t}^{j,n}}{ X_{t}^{i,n}- X_{t}^{j,n}}}\d t,  \qquad  X_0^n \= x_0^n \in \RR_+^n.
\]
The corresponding invariant measure, derived in \cite{JK20}, takes the form
\[
\d \pi_{\beta}^{n}(x) \;:=\; \frac{\1_{{(x_1,\ldots,x_n)\in\overline{\D}_n}} }{C_n^\beta}  \prod_{i=1}^n
x_i^{\parent{\alpha_n -(n-1)\beta}\frac{2}{\sigma^2}-1} e^{- \frac{2\lambda}{\sigma^2} x_i} \prod_{i > j}(x_i-x_j)^{\frac{4\beta}{\sigma^2}} \, \d x_i.
\]

\subsection{Known results about the cutoff}\label{ss:krc}

In this subsection, we gather key results from the literature, that allow us to deduce the existence of a cutoff phenomenon for the DL process.

\paragraph{Cutoff $\L^p$ with $p>1$} The following result follows from \cite[Corollary 3.4]{CSC08}. It asserts that, under a spectral condition and for a suitably chosen initial condition, the dynamics satisfy the so-called \textit{product condition}, which is sufficient to ensure the occurrence of a cutoff in $\L^p$-distance.
In other words, assume the initial condition $x_0^n$ satisfies
\begin{equation}\label{eq:iclp}
\frac{|\varphi_{n}(x_0^{n})|}{\sqrt{\alpha_n n + \frac{\beta}{2}(n-1)^2}} \xra{}{n\to\infty} +\infty,
\end{equation}
then, for all $p \> 1$, there exists a sequence of critical times $c_n \>0$ such that for every $\varepsilon \in (0,1)$,
\[
\lim_{n\to\infty} \norme{\Law( X^n_{t_n}) - \pi_{\beta}^{n}}{p} \=\begin{cases}
	+\infty & \text{if $t_n=(1-\varepsilon)c_n$},\\
	0 & \text{if $t_n=(1+\varepsilon)c_n$}.
\end{cases}
\]
To see this, consider $\varphi_{n}$, the eigenfunction \eqref{eq:efsp} associated with the eigenvalue $-1$ of the generator $\G_{\D\L}$ \eqref{eq:IG}. Then $\vabs{(\P_{t}^{n} - \pi_\beta^{n})(\varphi_{n})(x_0^{n})} \geq e^{-t} |\varphi_{n}(x_0^{n})|$, which provides a lower bound on the convergence in $\L^p$.
Moreover, since $\nabla \varphi_{n} = (1, \ldots, 1)$, one can compute the $\L^2(\pi_{\beta}^{n})$-norm via integration by parts
\[\norme{\varphi_{n}}{2}^2 \= - \ps{\G_{\D\L}\varphi_{n}, \varphi_{n}}_{2} \= \int \sum_i x_i (\partial_i \varphi_{n})^2 \, \d \pi_{\beta}^{n} \= \parent{\EE_{\pi_{\beta}^{n}}(\varphi_n) + \alpha_n n + \frac{\beta}{2}(n-1)^2},\]
and
\begin{equation}\label{eq:espvp}
\EE_{\pi_{\beta}^{n}}(\varphi_n) \= - \EE_{\pi_{\beta}^{n}}(1 \times \G\varphi_{n}) \=  \EE_{\pi_{\beta}^{n}}(\Gamma(1,\varphi_{n})) \=  \EE_{\pi_{\beta}^{n}}(x (\nabla 1) \cdot(\nabla \varphi_{n})) \= 0.
\end{equation}
Therefore, choosing $x_0^n$ as in \eqref{eq:iclp}, we deduce that the DL process undergoes a cutoff in $\L^2(\pi_{\beta}^{n})$-norm for the set of initial conditions \(\S_0^n = \{\delta_{x_0^n}\}\). Note that while the product condition is known to be sufficient for cutoff, its necessity, especially for general diffusions, remains an open question.

The main strength of this corollary is that it applies whenever the eigenfunction associated with the spectral gap is explicit. Moreover, the cutoff result extends to all \( p > 1 \) via interpolation. However, it does not yield an explicit expression for the mixing time, which remains the central focus of this work.

\paragraph{Cutoff for TV (for $p=1$)}
The previous result does not apply to the case \( p = 1 \), but it is still expected to hold for diffusions. A recent result in \cite{Sal25} shows that the product condition is also sufficient to ensure cutoff in total variation, i.e., for \( p = 1 \), provided the system has nonnegative curvature. This is the case for the DL process \eqref{eq:DL} (see Lemma \ref{lem:cd} for the curvature and Lemma \ref{lem:LBDL} for the product condition).

\section{Euclidean and Riemannian points of view of the Dyson--Laguerre process}

When dealing with a non-constant diffusion coefficient, as in \eqref{eq:IG}, two conceptual approaches naturally arise. On the one hand, if one is comfortable with geometric techniques, there is no need to change variables: the analysis proceeds directly in the intrinsic geometry of the process. On the other hand, adopting a Euclidean point of view, one may choose to perform a change of variables, thus absorbing the inhomogeneity of the diffusion coefficient and recovering a Euclidean structure amenable to classical tools.

\subsection{Euclidean point of view of the Dyson--Laguerre process} \label{sec:epov}

The change-of-variables strategy is both natural and effective here. It allows one to eliminate the non-constant nature of the diffusion coefficient  and thus remain within a Euclidean framework. For the Dyson--Laguerre process \eqref{eq:DL}, this transformation is performed via the square-root mapping $Y_{t}^{i,n} \;:=\; 2\sqrt{ X_{t}^{i,n}}$, leading to what is referred to as the \textit{Euclidean Dyson--Laguerre process} (EDL). We detail its Euclidean analysis in this section. In the matrix case, the problem reduces to the singular values of a rectangular OU matrix, which are natural observables.

The resulting SDE after the change of variables takes the form
\begin{equation}
\d  Y_{t}^{i,n} \= \sqrt{2} \d  B_{t}^{i,n} + \parent{\frac{2 \alpha_n-1}{ Y_{t}^{i,n}} - \frac{ Y_{t}^{i,n}}{2} + \frac{\beta}{ Y_{t}^{i,n}} \sum_{j \neq i} \frac{( Y_{t}^{i,n})^2+( Y_{t}^{j,n})^2}{( Y_{t}^{i,n})^2-( Y_{t}^{j,n})^2}} \d t,
\end{equation}
and the associated infinitesimal generator is 
\begin{equation}\label{eq:IGEDL}
\G_{\E\D\L}(f) \= \sum_{i=1}^{n} \partial_{ii}^2 f + (2\alpha_n-1)\sum_{i=1}^{n} \frac{1}{y_i} \partial_i f - \frac{1}{2} \sum_{i=1}^{n} y_i \partial_i f + \beta \sum_{i, j \neq i} \frac{1}{y_i} \frac{y_i^2+y_j^2}{y_i^2 - y_j^2}\partial_i f.
\end{equation}

A first observation is that the ground-state eigenfunction remains polynomial, as it is given
\[\widetilde{\varphi}_n(y) \= \sum_{i=1}^n y_i^2 - 4 n \alpha_n.\]
The interacting particle system described is no longer a Laguerre gas in the strict sense. Nevertheless, it retains enough structure to be analyzed. The invariant measure admits an explicit energy functional
\[
\widetilde{\E}(y) \;:=\; \sum_{i=1}^n \parent{\frac{1}{4} y_i^2 - \parent{\frac{\alpha_n}{2} - (n-1) \frac{\beta}{4} - \frac{1}{4}} \log(y_i) - \frac{\beta}{4} \sum_{j \neq i} \log(y_i^2-y_j^2)}.
\]
Notably, using that $\log(a^2-b^2) = \log(a - b) + \log (a+b)$ for $a>b>0$, the energy is convex. However, $\widetilde{\E}$ is not $\rho$-convex in the sense of \cite{CF24}. Indeed, the Hessian entries are
\[\Hess(\widetilde{\E})_{ii} \= \frac{1}{2} + \frac{\alpha_n - (n-1) \frac{\beta}{2} - \frac{1}{2}}{2} \frac{1}{y_i^2} + \beta \sum_{j \neq i} \frac{y_i^2+y_j^2}{(y_i^2-y_j^2)^2}, \quad \Hess(\widetilde{\E})_{ij} \= - \beta \frac{y_i y_j}{(y_i^2-y_j^2)^2}.\]
Consequently, the transformation preserves the log-concavity of the invariant measure.

Behind this transformation lies a broader question, namely the stability of the curvature-dimension condition $\C\D(\rho,\infty)$ condition under smooth diffeomorphisms. In the case of the EDL process, this property happens to be preserved. However, this equivalence should not be expected to hold in full generality. A computation of the carré du champ yields
\begin{equation}\label{eq:gammaEDL}
\Gamma^{\E\D\L}(f) \= \vabs{\nabla f}^2 \= \sum_{i=1}^{n} (\partial_{i} f)^2.
\end{equation}
Recall that in the Euclidean setting we have $\Gamma_{2}^{\E\D\L}(f) \= \norme{\Hess f}{2}^2 + (\nabla f)^\top (\Hess \, \widetilde{\E}) (\nabla f)$ for the computation of the $\Gamma_2$-operator, so that
\begin{equation}\label{eq:gamma2EDL}
\begin{aligned}
\Gamma_{2}^{\E\D\L}(f) &\=& \norme{\Hess f}{2}^2 + \frac{1}{2} \Gamma f + \frac{1}{2}\parent{\alpha_n - (n-1) \frac{\beta}{2} - \frac{1}{2}}\sum_{i} \frac{(\partial_i f)^2}{y_i^2}\\
	&& + \; \; \frac{\beta}{2} \sum_{j < i} \frac{(y_i (\partial_i f) - y_j (\partial_j f))^2}{(y_i^2-y_j^2)^2} + \frac{(y_i (\partial_j f) - y_j (\partial_i f))^2}
	{(y_i^2-y_j^2)^2}.
\end{aligned}
\end{equation}

We summarize below the key properties of the two processes under consideration. The comparison involves geometric and functional analytic features, including curvature-type bounds, contraction properties, and spectral structure.

\begin{table}[h!]
\centering
\begin{tabular}{@{} | l || c | c | @{}}
\toprule
\textbf{Property} & \textbf{Dyson--Laguerre Process} & \textbf{Euclidean Dyson--Laguerre Process} \\
\midrule
Diffusion coefficient & Square-root ($\propto \sqrt{x_i}$) & Constant \\
Non-interacting term & Gamma distribution & Normal distribution \\
Interacting term & $\beta \sum_{i>j} \log(x_i-x_j)$ & $\beta \sum_{i>j} \log(y_i^2-y_j^2)$ \\
Energy convexity & Convex & Convex \\
First eigenfunction & $\sum_{i} x_i$ + constant & $\sum_{i} y_i^2$ + constant \\
Curvature-dimension & $\C\D\parent{1/2, \infty}$ & $\C\D\parent{1/2, \infty}$ \\
\bottomrule
\end{tabular}
\caption{Comparison of properties between the DL and Euclidean DL processes.}
\end{table}
 
Comparing distances and divergences between the DL and EDL processes is relatively straightforward. For total variation, Kullback--Leibler, and $\L^2$ distances, we obtain equality by exploiting the contraction properties established in Lemma \ref{lem:contraction}. However, for the Wasserstein distance, the non-Lipschitz change of variables prevents a direct comparison between the Euclidean Wasserstein distances of the two processes.
Nevertheless, as discussed in Appendix \ref{sec:iwd}, the Wasserstein distance naturally adapted to the DL process is the intrinsic Wasserstein distance defined in \eqref{eq:iwd}. Importantly, this intrinsic distance evaluated along the DL process coincides with the standard Euclidean Wasserstein distance evaluated along the EDL process.
In particular, we note that the regularization inequality of Lemma \ref{lem:reg} can be recovered in this framework by following the Euclidean argument of \cite[Lemma 4.2]{BGL14}, followed by a comparison of relative entropy via contraction. This observation allows us to conclude that it is sufficient to study the EDL process to deduce the cutoff for the DL process with respect to all these distances.

\subsection{Riemannian point of view of the Dyson--Laguerre process} \label{sec:wgdl}

The non-Gaussian nature of this log-gas makes the derivation of functional inequalities in the standard (Euclidean) Wasserstein distance more delicate. In particular, one should not expect a (Euclidean) Talagrand-type inequality to hold in this setting, as such an inequality would imply sub-Gaussian concentration for functions that are Lipschitz with respect to the Euclidean metric, valid in the Gaussian case (see, e.g., \cite{BG99}) but not here.
Moreover, in the matrix setting, no Euclidean contraction inequality can be inferred: the spectral map sending a Wishart process to its eigenvalues (DL process) is not Lipschitz with respect to the Euclidean metric.

A more robust alternative to the previous section, which does not always ensure that the resulting process has a log-concave invariant law, is to use a Wasserstein distance intrinsically adapted to the geometry of the process. As detailed in Appendix \ref{sec:iwd}, this construction relies on the diffusion coefficients, which induce a Riemannian metric on a subset of $\RR^n$. In doing so, the state space inherits a geometric structure naturally aligned with the stochastic dynamics.

Recall that the diffusion coefficient of the Dyson--Laguerre process has components $\sigma : x_i \lms \sqrt{2 x_i}$, and that the corresponding carré du champ operator is \eqref{eq:cdc}.
In this setting, we consider the diagonal metric $g$ defined by $g_{ii}(x) \= 1/x_i$  and $g_{ij}(x) \= 0$ if $i\neq j$. This corresponds to the diagonal setting with the function $a : x \lms 1/\sqrt{x}$ on $\RR_+^\ast$. Since an antiderivative of $a$ is $A : x \lms 2 \sqrt{x}$, the associated Riemannian distance for the Dyson--Laguerre process reads
\begin{equation}
\fa x,y \in (\RR_+^\ast)^n \;:\quad d_{g}(x,y) \;:=\; 2 \sqrt{\sum_{i=1}^n \parent{\sqrt{x_i}-\sqrt{y_i}}^2}.
\end{equation}
Note that the geodesic $\gamma$ connecting $x$ to $y$ is given for $t\in[0,1]$ by $\gamma_i(t) \= \parent{t \sqrt{y_i} + (1-t)\sqrt{x_i}}^2$. The intrinsic Wasserstein distance of order $r\geq0$ is then defined for probability measures $\mu$ and $\nu$ on the same space by
\begin{equation}
\Wasserstein_{g,r}(\mu, \nu) \= 2\parent{\inf_{( X, Y)} \EE\intervalle{\parent{\sum_{i=1}^n \vabs{\sqrt{ X_i}-\sqrt{ Y_i}}^2}^{r/2}}}^{1/r}.
\end{equation}
where the infimum runs over all couplings $(X, Y)$ of $\mu$ and $\nu$.
We emphasize that the Lipschitz functions associated with this carré du champ operator are not linear: asymptotically for large \(x\), they behave like the square root of \(x\).

\section{Proof of the universal lower bound (Lemma \ref{lem:LBDL})}\label{sec:pulb}

In this section, we derive the universal lower bound by adapting the spectral method developed in \cite{SC94}, which estimates the $\L^2$-distance through spectral decomposition, and derives a total variation bound via the Bienaymé--Chebyshev inequality.

\begin{lem}[Universal lower bound]\label{lem:LBDL}
Let $(X^n_t)_{t\geq0}$ be the Dyson--Laguerre process \eqref{eq:DL} started at $x_0^n$. Then for $\dist \in \{ \T\V, \L^2, \Kullback, \Wasserstein_{}\}$
and for all $\varepsilon\in (0,1)$,
\[\lim_{n\to\infty}\dist(\Law( X^n_{(1-\varepsilon)c_{n}})\mid \pi_{\beta}^{n}) \= \max{}, \qquad\mtext{if}\quad c_{n} \;:=\; \log \parent{\frac{\vabs{\varphi_n(x_0^n)}}{\sqrt{\alpha_n n}}} \xra{}{n\to\infty} +\infty.\]
In particular, under these assumptions, the product condition is satisfied.
\end{lem}

\noindent Note that a lower bound for the DL process with interaction ($\beta>0$) may be obtained using contraction arguments, by projecting the dynamics onto a Cox--Ingersoll--Ross process. Indeed, defining \( Z_t := \phi_n(X_t)\), we get
\[
\d  Z_t = 2 \sqrt{ Z_t} \d  W_t + (\alpha_n n + \frac{\beta}{2} (n-1)^2 -  Z_t) \d t.
\]

\begin{proof}[Proof of Lemma \ref{lem:LBDL}]
We denote by $\mu_t^x := \delta_x \P_t$ and by $p_t^x = p_t(x, \cdot)$ the heat kernel density.

We begin by expanding the heat kernel in the orthonormal basis of $\L^2(\pi_{\beta}^{n})$, formed by the generalized Laguerre polynomials. Computing the $\L^2$ norm and retaining only the contribution from the spectral gap, we obtain a quantitative lower bound on the $\L^2$-distance. For the total variation distance, we follow the strategy of \cite[Lemma 2.2]{SC94}, which uses the Diaconis--Wilson lemma. Recall that $\norme{\varphi_n}{2}^2 = \alpha_n n + \frac{\beta}{2}(n-1)^2$.

For $\dist = \L^2$, recalling the computation of the $\L^2$ norm of the first eigenfunction in Subsection \ref{ss:krc}, we obtain the following lower bound
\[\norme{\mu_t^x - \pi_\beta^n}{2} \geq {\frac{\vabs{\varphi_n(x)}^2}{\norme{\varphi_n}{2}^2}} e^{-2t} \= \exp\parent{-2\parent{t - \frac{1}{2}\log \parent{\frac{\vabs{\varphi_n(x)}^2}{\alpha_n n + \frac{\beta}{2}(n-1)^2}}}}.\]
For the total variation and the Kullback--Leibler divergence, we apply Pinsker's inequality (see, e.g., \cite[Lemma A.1]{BCL})
\[\norme{\mu_t^x - \pi_\beta^n}{\T\V}^2 \leq 2 \Kullback(\mu_t^x \mid \pi_\beta^n),\]
and use \cite[Lemma 2.2]{SC94}
\[
\norme{\mu_t^x - \pi_\beta^n}{\T\V} \geq 1 - 4 \frac{\norme{\varphi_n}{2}^2}{\vabs{\varphi_n(x)}^2} e^{2t} - 4 \frac{\norme{\varphi_n}{2}^4}{\vabs{\varphi_n(x)}^4} \Var_{\mu_t^x}\parent{\frac{\varphi_n(x) \varphi_n(\cdot)}{\norme{\varphi_n}{2}^2} } e^{2t}.
\]
From Duhamel's formula (see, e.g., \cite[Equation (3.1.21)]{BGL14}) and that $\P_t \varphi_n = e^{-t} \varphi_n$
\begin{eqnarray*}
\Var_{\mu_t^x}(\varphi_n) = \P_t(\varphi_n^2)(x) - (\P_t \varphi_n(x))^2 &=& 2 \int_0^t \P_s (\Gamma (\P_{t-s}\varphi_n)) \d s\\
	&=& 2\varphi_n(x)(1-e^{-t})e^{-t} + (\alpha_n n + \frac{\beta}{2}(n-1)^2)(1-e^{-2t}).
\end{eqnarray*}
Hence, at the critical time $c_n$ defined above, the proof follows from the fact that
\[
\frac{\Var_{\mu_{c_n-\varepsilon}^x}(\varphi_n)}{\vabs{\varphi_n(x)}^2} \sim_{n\to\infty} \frac{\alpha_n n + \frac{\beta}{2}(n-1)^2}{\vabs{\varphi_n(x)}^2}.
\]
For the intrinsic Wasserstein distance, it suffices to bound it from below by using the Kullback--Leibler divergence and the regularization property established in Lemma \ref{lem:reg}.
\end{proof}

\section{Proof of the upper bound for the real matrix case $\beta =1$ (Theorem \ref{thm:cmc})}\label{sec:pubmc}

In this section, we prove Theorem \ref{thm:cmc}, namely the cutoff phenomenon when $\beta=1$. The matrix case parallels the approach taken in \cite{BCL}, as the DL process can be seen as the spectral image of a matrix-valued OU process, naturally associated with the Laguerre orthogonal ensemble (LOE) random matrix ensembles (\cite[Chapter 3]{For10}). The upper bound follows from contraction properties of the relevant distances, which reduce the analysis to the mixing time behavior of a tensorized OU process. In Appendix \ref{sec:mtou}, we derive the mixing time for OU processes with general coefficients, extending a result from \cite{BCL}.

We consider the case of random matrices, following the framework introduced in \cite{Bru89} namely, the Dyson--Laguerre process started at $x_0^n$, obtained as the spectrum of $(M_t)_{t \geq 0}$ the rectangular OU process \eqref{eq:rOU} whose invariant measure is $P^n$. The process then satisfies
\begin{equation}\label{eq:MC}
\d X_{t}^{i,n} \= \sqrt{2  X_{t}^{i,n}} \, \d  B^i_t + \parent{\frac{m}{2} -  X_{t}^{i,n} + \; \frac{1}{2} \sum_{j \neq i}\frac{ X_{t}^{i,n}+ X_{t}^{j,n}}{ X_{t}^{i,n}- X_{t}^{j,n}}}\d t, \qquad X_0^n \= x_0^n \in \RR_+^n,
\end{equation}
where $m\geq n$.

\begin{lem}[Upper bound in the Bru case]\label{lem:ubmc}
Let ${(X^n_t)}_{t\geq0}$ be the Dyson--Laguerre process \eqref{eq:MC} started at $x_0^n$ and resulting from $(M_t)_{t \geq 0}$. Then for $\dist \in \{\T\V, \Kullback, \L^2, \Wasserstein_{}\}$, and for all $\varepsilon\in (0,1)$
\[\lim_{n\to\infty}\dist(\Law( X^n_{(1+\varepsilon)c_n})\mid \pi_1^n) \= 0 \quad \mtext{with}\quad c_n \xra{}{n\to\infty} +\infty,\]
where $c_n$ is given by
\[c_n \= \begin{acc} $\dps \log\parent{\frac{\phi_n(x_0^n)}{m}} \vee \log\parent{n}$, & if $\dist \= \T\V$,\\
\\
$\dps \log\parent{\frac{\phi_n(x_0^n)}{m}} \vee \log \sqrt{n}$, & if $\dist \= \Kullback, \norme{.}{2}$,\\
\\
$\dps \log\parent{\phi_n(x_0^n)} \vee \log \sqrt{nm}$, & if $\dist \= \Wasserstein_{}$.
 \end{acc}
\]\end{lem}

To obtain such an upper bound, it suffices to control the distance between the process at time $t$ and its equilibrium. Given that the DL process is the image of a rectangular OU process, it is enough to estimate, at each time $t$, the distance to equilibrium of the underlying rectangular OU process. The following lemma establishes that this control is indeed possible.

\begin{lem}[From DL to rectangular OU]\label{lem:DLtoOU}
Under the assumptions of Lemma \ref{lem:ubmc},
\begin{eqnarray*}
	\norme{\Law(X_t^n) - \pi_1^n}{\T\V} &\leq& nm \norme{\Law( M_t) - P^n}{\T\V},\\
	\Kullback\parent{\Law(X_t^n) \mid \pi_1^n} &\leq& nm \Kullback\parent{\Law( M_t) \mid P^n},\\
	\norme{\Law(X_t^n) - \pi_1^n}{2}^2 &\leq& \parent{\norme{\Law( M_t) - P^n}{2}^2+1}^{nm} -1,\\
	\Wasserstein_{}(\Law(X_t^n), \pi_1^n) &\leq& \sqrt{n} \Wasserstein_{2}(\Law( M_t), P^n),
\end{eqnarray*}
where $\Wasserstein_{2}$ is the Euclidean Wasserstein distance.
\end{lem}

The natural approach is to exploit the contraction properties of the distances, as established in the following lemma. As mentioned above, the DL process is the image of a rectangular OU process under a measurable map, namely the spectral map defined in \eqref{eq:hw}. In fact, this map is measurable and, moreover, Lipschitz with respect to the Riemannian distance defined in Appendix \ref{sec:iwd}, which furnishes the geometric framework for contraction in the intrinsic Wasserstein distance.

\begin{lem}[Contraction properties]\label{lem:contraction}
Let $\mu$ and $\nu$ be two probability measures on the same measurable space $\S$, and let $f: \S \lms \M$ be a measurable function, where $\M$ is another measurable space. If $\dist \in \{\T\V,\Kullback,\L^2\}$, the following inequality holds
\[\dist(\nu \circ f^{-1}\mid \mu \circ f^{-1}) \leq \dist(\nu \mid \mu).\]
\end{lem}
The analogous statement for the intrinsic Wasserstein distance is given in Lemma \ref{lem:wcontraction} in Appendix \ref{sec:iwd}.

The contraction property of these distances follows from the fact that each admits a variational formulation. This perspective is adopted, for instance, in \cite{BCL}. For the intrinsic Wasserstein distance, we provide such a formulation in Appendix \ref{sec:iwd}. Variational representations also exist for the Fisher information (see, e.g., \cite{HM14}), although a version compatible with contraction properties would be required in the present context.

It remains to consider two coupled processes: one initialized from an arbitrary distribution, and the other from the invariant measure. Let $(M_t)_{t \geq 0}$ and $(\widehat{M}_t)_{t\geq0}$ be two matrices, whose entries evolve as independent OU processes, such that
\[\d  M_t \= \sqrt{\frac{m}{2}} \d  W_t - \frac{1}{2}  M_t \d t \quad\mtext{and}\quad \d \widehat{ M}_t \= \sqrt{\frac{m}{2}} \d  W_t - \frac{1}{2} \widehat{ M}_t \d t.\]

\begin{proof}[Proof of Lemma \ref{lem:DLtoOU}]
Denote by $\Lambda_i( M)$ the $i$-th singular value of a matrix $M$. By the Hoffman-Wielandt inequality, the mappings
\begin{equation}\label{eq:hw}
\deff{\Lambda_i:}{\MMM_{n,m}(\RR)}{M}{\RR}{\Lambda_i( M)},\end{equation}
are $1$-Lipschitz w.r.t. the Frobenius norm. Denoting by $\lambda_i( M)$ the $i$-th eigenvalue of $\sqrt{ M M^\top}$, it follows that
\[\lambda_i( M) \= \vabs{\Lambda_i( M)}.\]
Considering $ M \= ( M_t)_{t \geq 0}$, we finally obtain, for all $t\>0$
\begin{equation}\label{eq:proj} X_t^i \= \lambda_i( M_t)^2 \= \vabs{\Lambda_i( M_t)}^2.\end{equation}
Note that the functions $\vabs{\Lambda_i}^2$ are not Lipschitz but remain measurable. Therefore, applying Lemma \ref{lem:contraction}, for any $\dist \in \{ \T\V, \Kullback, \L^2\}$, we obtain
\begin{eqnarray*}
\dist\parent{\Law(X_t^i) \mid \Law(\widehat{ X}_t^i)} &\=& \dist\parent{\Law(\vabs{\Lambda_i( M_t)}^2) \mid \Law(|\Lambda_i(\widehat{ M}_t)|^2)}\\
	&\leq& \dist\parent{\Law( M_t) \mid \Law(\widehat{ M}_t)}.
\end{eqnarray*}
Using the tensorization property of these distances and divergences (see, e.g., \cite[Lemma A.4]{BCL}), we deduce the following.

\noindent For $\dist \in \{ \T\V, \Kullback\}$
\[\dist\parent{\Law(X_t) \mid \Law(\widehat{ X}_t)} \leq nm \dist\parent{\Law( M_t) \mid \Law(\widehat{ M}_t)}.\]
For $\dist \= \L^2$, we have
\[\norme{\Law(X_t) - \Law(\widehat{ X}_t)}{2}^2 \leq -1 + \parent{\norme{\Law( M_t) - \Law(\widehat{ M}_t)}{2}^2+1}^{nm}.\]
For $\dist \= \Wasserstein_{}$, since each $\Lambda_i$ with $i\in\bracket{1,n}$ is $1$-Lipschitz, we observe from \eqref{eq:proj} that
\[\vabs{\sqrt{X_t^i}-\sqrt{\widehat{ X}_t^i}} \= \vabs{\vabs{\Lambda_i( M_t)} - \vabs{\Lambda_i(\widehat{ M}_t)}} \leq \vabs{\Lambda_i( M_t) - \Lambda_i(\widehat{ M}_t)}\leq \vabs{ M_t - \widehat{ M}_t},\]
which implies
\[\Wasserstein_{}(\Law(X_t), \Law(\widehat{ X}_t)) \leq 2\sqrt{n} \, \Wasserstein_{2}(\Law( M_t), \Law(\widehat{ M}_t)).\]
Note that the Frobenius norm on the space of real $n \times m$ matrices coincides with the Euclidean norm on $\RR^{nm}$.
\end{proof}

Finally, we state the mixing time for a rectangular OU process with general coefficients in the following lemma. This result is straightforward since the distance to equilibrium for the OU process can be computed explicitly; details are provided in Appendix \ref{sec:mtou}.

\begin{lem}[Mixing time of the rectangular Ornstein--Uhlenbeck process]\label{lem:tmixeqOU}
Under the assumptions of Lemma \ref{lem:ubmc}, the sequence $(M_t)_{t \geq 0}$ exhibits cutoff, and its mixing time is given for any $z\in\RR^{nm}$ by
\[\tmix^n(\{z\}) \= \begin{acc} $\dps \log\parent{\frac{\vabs{z}^2}{m}} \vee \log\parent{n}$, & if $\dist \= \T\V$,\\
\\
$\dps \log\parent{\frac{\vabs{z}^2}{m}} \vee \log \sqrt{n}$, & if $\dist \in \{\Kullback, \norme{.}{2}\}$,\\
\\
$\dps \log\parent{\vabs{z}^2} \vee \log \sqrt{nm}$, & if $\dist \= \Wasserstein_{}$.
 \end{acc}\]
\end{lem}

\begin{proof}[Proof of Lemma \ref{lem:tmixeqOU}]
The result follows directly from combining Lemmas \ref{lem:tmixOU} and \ref{lem:efROU}.
\end{proof}

By combining the last three lemmas, we readily obtain the upper bound stated above. This shows that the upper bound, and consequently the mixing time given in Theorem \ref{thm:cmc}, actually coincides with the mixing time of a rectangular OU process.

\begin{proof}[Proof of Lemma \ref{lem:ubmc}]
Combining Lemmas \ref{lem:DLtoOU} and \ref{lem:tmixeqOU}, and considering $ X_0 = x_0$ the initial condition, where $x_0^i \=  X_0^i \= \vabs{\Lambda_i( M_0)}^2$, we may choose the initial matrix $M_0$ to be diagonal with diagonal coefficients $z_0^1, \hdots, z_0^n$, so that $x_0^i \= (z_0^i)^2$ and $|z_0|^2 \= \sum_{i=1}^n (z_0^i)^2 \= \sum_{i=1}^n x_0^i  \;=:\; \phi_n(x_0)$.
\end{proof}

It remains only to combine the lower and upper bounds to conclude the result in the matrix case.

\begin{proof}[Proof of Theorem \ref{thm:cmc}]
From Lemma \ref{lem:LBDL}, the lower bound is $\log\parent{\frac{\phi_n(x_0)}{\sqrt{nm}}} \vee \log \sqrt{nm}$. The conclusion then follows from Lemma \ref{lem:ubmc}.
\end{proof}

\section{Proof of the upper bound in the general case (Theorem \ref{thm:CDL})}\label{sec:pubgc}

In this final section, we establish the upper bound stated in the following lemma. Our approach combines curvature arguments with a regularization technique from Lemma \ref{lem:reg}.

\begin{lem}[Universal upper bound]\label{lem:UBDL}
Let ${( X^n_t)}_{t\geq0}$ be the Dyson--Laguerre process \eqref{eq:DL} started at $x_0^n$. Then for any $\dist \in \{ \T\V, \Kullback, \Wasserstein_{}\}$ and all $\varepsilon\in (0,1)$
\[\lim_{n\to\infty}\dist(\Law( X^n_{(1+\varepsilon)c_n})\mid \pi_\beta^n) \= 0 \quad\mtext{with}\quad c_{n} \;:=\; \log\parent{\phi_n(x_0^n)} \vee \log \parent{\alpha_n n} \xra{}{n\to\infty} + \infty.\]
\end{lem}

\subsection{Curvature-dimension of the process}

The Dyson--Ornstein--Uhlenbeck process studied in \cite{BCL} evolves according to Langevin dynamics driven by a convex energy, which makes the curvature condition immediate to verify. Indeed, such a condition is naturally expected for Langevin diffusions with convex potentials.
In our setting, however, the situation is more subtle. The process is not a gradient diffusion. Nonetheless, it is linked to the Coulomb gas \eqref{eq:gc}, where the interaction is logarithmic and thus convex. This suggests curvature, but the absence of a gradient structure prevents a direct application of standard results, and it is not clear whether a curvature-dimension condition can hold. The next lemma shows that such a curvature condition does in fact hold. The one-dimensional case, studied in \cite{LNP}, plays a guiding role and informs the structure of our argument.
Following the approach of \cite[Section 3]{DGZ24}, one may directly verify the symmetric tensor identity (C.6.3) from \cite[Section C.6]{BGL14}, which yields the geometric curvature associated with the process.

\begin{lem}[Curvature-dimension]\label{lem:cd}
The Dyson--Laguerre process \eqref{eq:DL} satisfies the curvature-dimension condition $\C\D(1/2,\infty)$.
\end{lem}

\begin{proof}[Proof of Lemma \ref{lem:cd}]
Recall that the carré du champ operator is given by \eqref{eq:cdc}.
The $\Gamma_2$-operator takes the following form. For $\delta_n := \alpha_n - (n-1)\frac{\beta}{2}$,
\begin{eqnarray*}
\Gamma_2 f &\=& \frac{\delta_n}{2} \vabs{\nabla f}^2 + \frac{1}{2}\Gamma(f) +  \sum_{i} \parent{x_i^2 (\partial^2_{ii} f)^2 + x_i(\partial_i f) (\partial_{ii}^2 f)}\\
	& & +  \sum_{i} \sum_{j < i} \parent{2 x_i x_j (\partial^2_{ij} f)^2 + \frac{\beta}{2} \frac{x_i^2(\partial_i f)^2 + x_j^2(\partial_j f)^2}{(x_i-x_j)^2}}\\
	& & +  \frac{\beta}{2} \sum_i \sum_{j < i} \frac{x_i x_j}{(x_i-x_j)^2} \parent{(\partial_i f)^2 - 4 (\partial_i f)(\partial_j f) + (\partial_j f)^2},
\end{eqnarray*}
(see below for a detailed computation).

Now, using the inequality $x_i(\partial_i f) (\partial_{ii}^2 f) \geq - \frac{1}{2} \parent{x_i^2 (\partial^2_{ii} f)^2 + (\partial_i f)^2}$,
we obtain
\[\frac{\delta_n}{2} \vabs{\nabla f}^2 + \frac{\Gamma(f)}{2} +  \sum_{i} \parent{x_i^2 (\partial^2_{ii} f)^2 + x_i(\partial_i f) (\partial_{ii}^2 f)}\geq \frac{\delta_n-1}{2} \vabs{\nabla f}^2 + \frac{\Gamma(f)}{2} + \frac{1}{2} \sum_{i} x_i^2 (\partial^2_{ii} f)^2.\]
In addition, from the elementary inequality $x_i^2(\partial_i f)^2 + x_j^2(\partial_j f)^2 \geq 2 x_i x_j (\partial_i f) (\partial_j f)$, we deduce
\[ \beta \sum_{i} \sum_{j < i} \frac{x_i^2(\partial_i f)^2 + x_j^2(\partial_j f)^2}{(x_i-x_j)^2} \geq \beta \sum_{i} \sum_{j < i} \frac{2 x_i x_j (\partial_i f) (\partial_j f)}{(x_i-x_j)^2}.\]
Combining the above, we conclude
\begin{eqnarray*}
 \beta \sum_{i} \sum_{j < i} \frac{x_i^2(\partial_i f)^2 + x_j^2(\partial_j f)^2}{(x_i-x_j)^2} &+& \beta \sum_i \sum_{j < i} \frac{x_i x_j}{(x_i-x_j)^2} \parent{(\partial_i f)^2 - 4 (\partial_i f)(\partial_j f) + (\partial_j f)^2}\\
	& \geq & \beta \sum_i \sum_{j < i} \frac{x_i x_j}{(x_i-x_j)^2} \parent{(\partial_i f)^2 - 2 (\partial_i f)(\partial_j f) + (\partial_j f)^2}\\
	& \= & \beta \sum_i \sum_{j < i} \frac{x_i x_j}{(x_i-x_j)^2} \parent{(\partial_i f) - (\partial_j f)}^2.
\end{eqnarray*}
\end{proof}

\begin{proof}[Justification of the $\Gamma_2$ formula] The computation of \( \Gamma_2 \) follows closely the strategy outlined in \cite{LNP}.
Recall that the carré du champ is given by \eqref{eq:cdc}, and that
\[\G f \= \sum_{i} x_i \partial_{ii}^2 f + \sum_{i} b_i \partial_i f \qquad\mtext{where}\qquad b_i(x) \= \alpha_n -  x_i + \frac{\beta}{2}\sum_{k\neq i} \frac{x_i+x_k}{x_i-x_k}.\]
First, we compute
\begin{eqnarray*}
\frac{1}{2}\G\Gamma f &\=& \frac{1}{2}\sum_{i} \left(4 x_i (\partial_i f)(\partial_{ii}^2 f) + \sum_{j} 2 x_i x_j \parent{(\partial_{ij}^2 f)^2 + (\partial_j f) (\partial_{iij}^3 f)}\right.\\
	&& \qquad \left.+ b_i (\partial_i f)^2 + \sum_{j} 2 x_j b_i (\partial_j f)(\partial_{ij}^2 f)\right).
\end{eqnarray*}
Then we have
\begin{eqnarray*}
\Gamma (f, \G f) &\=& \frac{1}{2}\sum_{j} \left(2x_j(\partial_j f) (\partial_{jj}^2 f) + \sum_{i} 2x_i x_j (\partial_j f) (\partial_{iij}^3 f) \right.\\
	&& \qquad \left. + \sum_{i} 2x_j (\partial_j f) \parent{(\partial_j b_i)(\partial_i f) + b_i (\partial_{ij}^2 f)} \right).
\end{eqnarray*}
Thus, we obtain
\begin{eqnarray*}
\Gamma_2 f &\=& \frac{1}{2}\G\Gamma f - \Gamma(f,\G f)\\
	&\=&\frac{1}{2} \sum_{i} \parent{\sum_j 2 x_i x_j (\partial^2_{ij} f)^2 + 2 x_i(\partial_i f) (\partial_{ii}^2 f) + b_i (\partial_i f)^2 - \sum_j 2 x_i (\partial_i f)(\partial_j f)(\partial_i b_j)}\\
	&\=& \frac{1}{2} \sum_{i} \left(\sum_j 2 x_i x_j (\partial^2_{ij} f)^2 + 2 x_i(\partial_i f) (\partial_{ii}^2 f) + (b_i - 2 x_i (\partial_i b_i)) (\partial_i f)^2\right.\\
	& & \qquad \left. - \sum_{j\neq i} 2 x_i (\partial_i f)(\partial_j f)(\partial_i b_j) \right) \\
	&\=& \frac{1}{2} \sum_{i} \parent{2 x_i^2 (\partial^2_{ii} f)^2 + 2 x_i(\partial_i f) (\partial_{ii}^2 f) + (b_i - 2 x_i (\partial_i b_i)) (\partial_i f)^2}\\
	& & \qquad + \frac{1}{2} \sum_{i} \sum_{j\neq i} \parent{2 x_i x_j (\partial^2_{ij} f)^2 - 2 x_i (\partial_i f)(\partial_j f)(\partial_i b_j)}.\\
	&\=& \frac{1}{2} \sum_{i} 2 x_i^2 (\partial^2_{ii} f)^2 + 2 x_i(\partial_i f) (\partial_{ii}^2 f) + (b_i - 2 x_i (\partial_i b_i)) (\partial_i f)^2 \\
	& & +  \sum_{i} \sum_{j < i} \parent{2 x_i x_j (\partial^2_{ij} f)^2 - (\partial_i f)(\partial_j f)(x_i \partial_i b_j + x_j \partial_j b_i)}.
\end{eqnarray*}
Now, using that
\[\partial_i b_i(x) \= - 1 - \beta \sum_{k\neq i} \frac{x_k}{(x_i-x_k)^2}, \quad \partial_i b_j(x) \= \beta \frac{x_j}{(x_i-x_j)^2},\]
we deduce that, for each index $i$
\[b_i - 2 x_i (\partial_i b_i) \= \delta_n + x_i + \beta \sum_{k\neq i} \frac{x_i^2 + x_i x_k}{(x_i-x_k)^2},\]
and,
\[\fa j < i \;:\quad x_i \partial_i b_j + x_j \partial_j b_i \= \beta \parent{\frac{x_i x_j}{(x_i-x_j)^2} + \frac{x_i x_j}{(x_i-x_j)^2}} \= 2 \beta \frac{x_i x_j}{(x_i-x_j)^2}
.\]
Note that $\sum_{i, j> i} a_{ij} \= \sum_{i, j<i} a_{ji}$, which allows us to symmetrize the expression. As a consequence, we obtain
\begin{eqnarray*}
\frac{1}{2} \sum_{i} (b_i - 2 x_i (\partial_i b_i)) (\partial_i f)^2 &\=& \frac{\delta_n}{2} \vabs{\nabla f}^2 + \frac{1}{2} \Gamma(f)\\
	&& +  \frac{\beta}{2} \sum_i \sum_{j < i} \frac{x_i^2(\partial_i f)^2 + x_j^2(\partial_j f)^2}{(x_i-x_j)^2} + \frac{x_i x_j((\partial_i f)^2 + (\partial_j f)^2)}{(x_i-x_j)^2},
\end{eqnarray*}
and
\[\sum_{i} \sum_{j < i} (\partial_i f)(\partial_j f)(x_i \partial_i b_j + x_j \partial_j b_i) \= 2 \beta \sum_{i} \sum_{j < i} (\partial_i f)(\partial_j f) \frac{x_i x_j}{(x_i-x_j)^2}.\]
This completes the proof.
\end{proof}

\subsection{Sub-exponential convergence and regularization}\label{sec:secr}

Lemma \ref{lem:cd} yields two distinct types of information, each with its own implications.
First, it implies a family of functional inequalities, such as the Poincaré and logarithmic Sobolev inequalities. These in turn yield quantitative convergence results, including sub-exponential decay of the $\L^2$ norm and the relative entropy (see \cite{Ane00, BGL14}).
Second, the curvature condition provides a bound on the carré du champ operator \(\Gamma\) via the \(\Gamma_2\) operator. This control enables the derivation of regularization estimates for the relative entropy in terms of the intrinsic Wasserstein distance, following the strategy of \cite{BGL}.

\paragraph{Functional inequalities and sub-exponential convergence}
It is not obvious \emph{a priori} that a logarithmic Sobolev inequality should hold for the DL process. This misconception likely stems from the fact, already mentioned, that the process is not Langevin type, unlike in the OU case. 
Nevertheless, Lemma \ref{lem:cd} shows that both the Poincaré and log-Sobolev inequalities are indeed satisfied. More precisely, the Poincaré inequality holds with constant 2, and the log-Sobolev inequality with constant $4$.

As a consequence, we obtain the following relaxation estimates: for all $t\geq0$, the process satisfies the sub-exponential decay bounds
\begin{align*}
\norme{\Law( X^n_t) - \pi_{\beta}^{n}}{2}^2 &\leq e^{- t}\norme{\Law( X^n_0) - \pi_{\beta}^{n}}{2}^2,\\
\Kullback(\Law( X^n_t)\mid \pi_{\beta}^{n}) &\leq e^{- t}\Kullback(\Law( X^n_0)\mid \pi_{\beta}^{n}).
\end{align*}
Note that in each case, if the right-hand side is infinite, then the inequality holds trivially. This occurs, for instance, for the relative entropy if the initial law $\Law( X^n_0)$ is not absolutely continuous with respect to the invariant measure.

Finally, to establish the exponential decay of the intrinsic Wasserstein distance, we refer to Lemma \ref{lem:expdec} in Appendix \ref{sec:iwd}. The general framework for such results can be found in \cite{vRS05}.

\paragraph{Monotonicity}
Recall that the notion of mixing time is well-defined as soon as the following monotonicity property holds: for some sufficiently large $n$, and for any $\dist \in \{\T\V, \Kullback, \L^2,$ $\Wasserstein_{}\}$, the map $t\geq0\mapsto\dist(\Law( X^n_t)\mid \pi_{\beta}^{n})$ is non-increasing.

Establishing this monotonicity is non-trivial in general. For $\Phi$-entropies, such as relative entropy, total variation, or the $\L^2$ norm, the question is addressed in \cite{Cha04}, which shows that the monotonicity holds for a wide class of diffusion processes. In contrast, monotonicity for the intrinsic Wasserstein distance is more subtle (see, e.g., \cite{BGG12}). However, in our setting, the exponential decay is sufficient to conclude that monotonicity holds.

\subsection{End of the proof}
We may then appeal to the regularization result from Lemma \ref{lem:reg}, noting that while the Kullback--Leibler divergence may be infinite (e.g., for singular initial laws), the intrinsic Wasserstein distance remains finite in our case.

\begin{proof}[Proof of Lemma \ref{lem:UBDL}]
Denoting $\sqrt{x_0^n} = (\sqrt{x_0^{n,i}})_i$, the triangle inequality for the intrinsic Wasserstein distance yields
\[\frac{1}{4}\Wasserstein_{}^2(\delta_{x^n_0}, \pi_{\beta}^{n}) \= \int\vabs{\sqrt{x^n_0} - \sqrt{x}}^2 \d \pi_{\beta}^{n}(x) \leq 2 \phi_n(x^n_0) + \int 2\phi_n(x) \d \pi_{\beta}^{n}(x),\]
where $\EE_{\pi_{\beta}^{n}}(\phi_n)  \= \alpha_n n + \frac{\beta}{2}(n-1)^2$. It then follows that, for every $\eta > 0$,
\begin{eqnarray*}
\Kullback(\Law( X_{t+\eta}^n), \pi_{\beta}^{n}) & \leq & \frac{e^{-\eta}}{4(1-e^{-\eta})} \cdot \Wasserstein^2(\Law(X_t^n), \pi_{\beta}^{n})\\
	& \leq & \frac{e^{-\eta}}{1-e^{-\eta}} \cdot \frac{1}{4} \Wasserstein^2(\Law( X_0^n), \pi_{\beta}^{n}) \, e^{-t}\\
	&\leq& \frac{e^{-\eta}}{1-e^{-\eta}} \parent{\phi_n(x^n_0) + \alpha_n n + \frac{\beta}{2}(n-1)^2} e^{-t}.
\end{eqnarray*}
\end{proof}

\begin{proof}[Proof of Theorem \ref{thm:CDL}]
In \cite[Section 3]{CSC08}, it is shown that for the $\L^p$ norm, once $(i)$ monotonicity, $(ii)$ sub-exponential convergence, and $(iii)$ the product condition hold, the process exhibits a cutoff. This extends to any distance or divergence provided it satisfies those three properties.
The lower bound in Lemma \ref{lem:LBDL} yields the product condition.
Section \ref{sec:secr} and Lemma \ref{lem:expdec} yield monotonicity and sub-exponential convergence.
Hence the process \eqref{eq:DL} exhibits cutoff with respect to the $\L^2$ norm, Kullback--Leibler divergence, and the intrinsic Wasserstein distance.
For total variation distance, one may apply the criterion of \cite[Corollary 1]{Sal25} directly using the product condition.

Finally, to bound on the mixing time, combine the lower bound in Lemma \ref{lem:LBDL} with the upper bound in Lemma \ref{lem:UBDL}.
\end{proof}

\begin{appendix}

\section{Mixing times for Ornstein--Uhlenbeck processes% under various distances
}\label{sec:mtou}

The following lemma is a variant of \cite[Theorem 1.2]{BCL}, adapted to Ornstein--Uhlenbeck processes with general coefficients. It will be applied in Section~\ref{sec:pubmc} to the matrix case, where the DL process arises as the image of such an OU process.

\begin{lem}[Mixing time for Ornstein--Uhlenbeck processes]\label{lem:tmixOU}
Fix $\sigma_n^2 \>0$ and $\theta_n \in \RR$. Define $(Z_t^{n})$ by
\[
\d Z_t^{n} = \sigma_n \, \d B_t^{n} - \theta_{n}  Z_t^{n} \, \d t, \qquad  Z_0^{n} = z_0^n.
\]
Then the family exhibits a cutoff phenomenon, and its mixing time satisfies
\[2\theta_n\tmix^n(\{z_0^n\}) = \begin{acc} $\dps \log\parent{\frac{\theta_n}{4\sigma_n^2} \vabs{z_0^n}^2} \vee \log\parent{\frac{n}{4}}$, & if $\dist = \T\V$,\\
\\
$\dps \log\parent{\frac{\theta_n}{\sigma_n^2} \vabs{z_0^n}^2} \vee \log\parent{\frac{\sqrt{n}}{2}}$, & if $\dist = \Kullback$,\\
\\
$\dps \log\parent{\frac{2\theta_n}{\sigma_n^2} \vabs{z_0^n}^2} \vee \log\sqrt{\frac{n}{2}}$, & if $\dist = \norme{.}{2}$,\\
\\
$\dps \log\parent{\vabs{z_0^n}^2} \vee \log \sqrt{\frac{n \sigma_n^2}{8\theta_n}}$, & if $\dist = \Wasserstein_{2}$,
 \end{acc}\]
where $\Wasserstein_{2}$ denotes the standard $2$-Wasserstein (Euclidean) distance.
\end{lem}

These mixing times follow from explicit computations, relying on closed-form expressions for various distances and divergences between Gaussian distributions. The result, adapted to the rectangular Ornstein--Uhlenbeck setting, is stated in the following lemma and extends \cite[Lemma A.5]{BCL}.

\begin{lem}[Explicit formulas for the rectangular Ornstein--Uhlenbeck process] \label{lem:efROU}
Let $(M_t)_{t \geq 0}$ be a matrix-valued process whose entries evolve as independent OU processes, $\d  M_t = \kappa \d  B_t - \gamma  M_t \d t$, started from $M_0 = z_0$ and with invariant law $P$.
Then
\begin{eqnarray*}
	2\Kullback\parent{\Law( M_t) \mid P} & = & \frac{2 \gamma}{\kappa^2} |z_0|^2 e^{-2\gamma t} - nm e^{- 2 \gamma t} - nm \log(1 - e^{- 2 \gamma t}),\\
	\norme{\Law( M_t) - P}{2}^2 & = & -1 + \exp\parent{\frac{2\gamma}{\kappa} \frac{|z_0|^2 e^{-2 \gamma t}}{1 + e^{- 2 \gamma t}} - \frac{nm}{2}\log\parent{1 - e^{- 4 \gamma t}}},\\
	\Wasserstein_{2}^2\parent{\Law( M_t), P} & = & |z_0|^2 e^{-2\gamma t} + nm \frac{\kappa^2}{2 \gamma} \parent{1 - \sqrt{1 - e^{- 2 \gamma t}}}.\\
\end{eqnarray*}
\end{lem}

\section{Intrinsic Wasserstein distance}\label{sec:iwd}

The purpose of this section is to define a Wasserstein-type distance adapted to diffusion processes with non-constant diffusion coefficients. The objective is to retain, as much as possible, the essential properties of the classical (Euclidean) Wasserstein distance, observed in the study of Langevin diffusions.

While the log-Sobolev inequality is typically studied in the setting of Langevin diffusions with constant diffusion, the process considered here features a space-dependent diffusion coefficient. Nevertheless, we will see that the Bakry-Émery method remains applicable: the diffusion coefficient plays the role of a curvature term. From a geometric perspective, this is consistent-- the diffusion matrix can be interpreted as a Riemannian metric on the state space $\RR^n$, thereby inducing a geometry tailored to the process (see \cite[Subsubsection 5.3.2]{Ane00}, \cite[Appendix C]{BGL14} for more details). 

This perspective leads us to define a class of Lipschitz functions adapted to the Riemannian metric and, in turn, a transportation cost that respects the underlying geometry. The resulting distance, referred to as the \textit{intrinsic Wasserstein distance}, reflects the structure imposed by the diffusion. Under the curvature-dimension condition $\C\D(\rho,\infty)$, we show exponential convergence to equilibrium in this intrinsic metric.

\subsection{State space as a Riemannian manifold}\label{sec:ssrm}

In this subsection, we set up the basic geometric framework to handle a non-constant diffusion coefficient. We then examine the Riemannian distance induced by this metric. In the particular case where the metric is diagonal, we derive the explicit form of the geodesic equation, which will be useful later. Finally, we consider Lipschitz functions with respect to the Riemannian distance, noting that their Lipschitz constant is the supremum norm of the carré du champ, as shown in \cite[Section 3.2]{Hir05}.

\paragraph{Metric and Riemannian distance}
Let $L$ be a diffusion operator on a smooth $n$-dimensional manifold $\M$, which in local coordinates takes the form
\begin{equation}\label{eq:dol}
    L \= \sum_{i,j=1}^{n} g^{ij}(x) \partial^2_{i j} + \sum_{i=1}^{n} b^i(x) \partial_i,
\end{equation}
where the symmetric matrix $(g^{ij})_{i,j} := \frac{1}{2}\sigma \sigma^\ast$ is positive-definite. Its inverse $g = (g_{ij})_{i,j}$ defines a Riemannian metric, thus endowing $\M$ with the structure of a Riemannian manifold $(\M, g)$.
Moreover, for any smooth function $f \in C^\infty(\M)$, the carré du champ operator is
\begin{equation}
\Gamma(f)(x) \= \sum_{i,j=1}^{n} g^{ij}(x) \partial_i f(x) \partial_j f(x) \= \vabs{\nabla_g f}_g^2.
\end{equation}
The Riemannian distance $d_g$ between two points $x,y \in \M$ is defined as
\[d_{g}(x,y) \;:=\; \inf_{\gamma \in \CCC_{x,y}} \int_0^1 \vabs{\gamma'(t)}_{\gamma(t)} \d t,\]
where $\CCC_{x,y}$ denotes the set of piecewise smooth curves from $[0,1]$ to $\M$ with $\gamma(0) = x$ and $\gamma(1)=y$.

\paragraph{Lipschitz functions with non-constant metric} By interpreting the state space $\M$ as a Riemannian manifold, we use a notion of Lipschitz functions adapted to the non-constant metric induced by the diffusion.

Let $\Lip(f)$ denote the Lipschitz constant of a function $f$ on $\M$ with respect to $d_g$, that is
\[\Lip(f) \;:=\; \sup_{\substack{x,y \in \M \\ x\neq y}} \frac{\vabs{f(x)-f(y)}}{d_g(x,y)}.\]
The Riemannian distance then admits a dual representation in terms of Lipschitz functions
\[\fa x,y \in \M \;:\quad d_{g}(x,y) \= \sup_{\text{Lip}(f) \leq 1} \vabs{f(x)-f(y)}.\]

In the Euclidean case, Rademacher's theorem implies that Lipschitz functions are characterized by having bounded gradients. In the Riemannian setting, an analogous description involves the carré du champ operator $\Gamma$, which links to Dirichlet forms. This approach was formalized by Nik Weaver \cite{Wea00} and refined by Francis Hirsch \cite{Hir05}.
Following \cite[Sections C.4 and 3.3.7]{BGL14}, one obtains $\text{Lip}(f) \= \norme{\sqrt{\Gamma f}}{\infty}$, and thus the Riemannian distance can be equivalently written as 
\begin{equation}\label{eq:dgamma}
d_g(x,y) \= \sup_{\norme{\Gamma f}{\infty} \leq 1} \vabs{f(x)-f(y)}.
\end{equation}

\paragraph{Diagonal case}
In this case, we consider a product structure for the diffusion coefficient, that is, a setting in which, for each particle, the diffusion coefficient depends only on the particle itself. We restrict our attention to the case where, in the absence of interaction (present only in the drift), the system consists of $n$ independent and identically distributed ergodic particles. Since $\frac{1}{2}\sigma \sigma^\ast$ is a positive-definite matrix, we consider a strictly positive smooth function $a$ and set
\begin{equation}\label{eq:mdc}
g_{ij}(x) \= \begin{acc}
	$a(x_i)^2$, & if $i=j$,\\
	$0$, & if $i \neq j$.
\end{acc}
\end{equation}
Note that $g^{ii}(x) \= a(x_i)^{-2}$. The corresponding Christoffel symbols are then given, for all $k,i,j$, by
\[\Gamma_{ij}^k := \frac{1}{2} \sum_{\ell\=1}^n g^{k \ell} \parent{\partial_i g_{j\ell} + \partial_j g_{i \ell} - \partial_{\ell} g_{ij}} \=
\begin{acc}
	$\frac{1}{2} g^{ii} \partial_i g_{ii} \= \frac{a'(x_i)}{a(x_i)}$, & if $i=j=k$,\\
	$0$ & else.
\end{acc}\]
From this, the geodesic equation reads for all $i$
\[\gamma_i'' + \frac{1}{2} g^{ii} \partial_i g_{ii} (\gamma_i')^2 \= 0 \quad\Llra\quad
\frac{\gamma_i''}{\gamma_i'} + \frac{a'(\gamma_i)\gamma_i'}{a(\gamma_i)} \= 0 \quad\Llra\quad (\log \gamma_i' + \log a(\gamma_i))' \= 0.\]
Hence, for each $i$, the quantity $a(\gamma_i(t))\gamma_i'(t)$ is conserved over time; denote this constant by $c_i$. 
The problem then reduces to solving a first-order autonomous ODE with boundary conditions: for each $i$, $\gamma_i$ satisfies
\[
\begin{acc}
	$\dps \gamma_i' \= c_i / a (\gamma_i)$,\\
	$\gamma_i(0) \= x_i$,\\
	$\gamma_i(1) \= y_i$.
\end{acc}
\]
Let $A$ be an antiderivative of $a$, which is strictly increasing and hence invertible on its range. Then standard computations yield the explicit formula
\begin{equation}\label{eq:gef}
\gamma_i(t) \= A^{-1}\intervalle{A(y_i) t + A(x_i) (1-t)}.
\end{equation}
and in particular, $c_i \= A(y_i) - A(x_i)$.
Also, since
\[\vabs{\gamma'(t)}_{\gamma(t)}^2 \= \sum_{i=1}^{n} g_{ii}(\gamma)(\gamma'_{i})^2 \= \sum_{i=1}^n a(\gamma_i)^2 (\gamma'_{i})^2 \= \sum_{i=1}^n \vabs{A(y_i) - A(x_i)}^2,\]
and $d_{g}(x,y) \= \int_0^1 \vabs{\gamma'(t)}_{\gamma(t)} \d t$, which simplifies to
\begin{equation}\label{eq:dA}
d_{g}(x,y) \= \sqrt{\sum_{i=1}^n \vabs{A(y_i) - A(x_i)}^2}.
\end{equation}
As a consequence, we have the identity
\begin{equation}\label{eq:ngpd}
\fa t\in[0,1],\quad \vabs{\gamma'(t)}_{\gamma(t)} \= d_{g}(x,y).
\end{equation}

This recovers the Euclidean case, where $g$ is the identity matrix, $A$ is the identity map on $\RR$, and the distance is the Euclidean norm $d_{g = \I_n}(x,y) \= \vabs{x-y}$, where $x,y \in \RR^n$.

\subsection{Intrinsic Wasserstein distance}

This intrinsic nature of the Wasserstein distance is entirely induced by the Riemannian structure defined earlier. In what follows, we consider two probability measures $\mu$ and $\nu$ on $\M$. This distance is defined as the optimal transport cost associated with the Riemannian distance $d_g$, for any $r \geq 0$,
\[\Wasserstein_{g,r}(\mu, \nu) \;:=\; \parent{\inf_{\pi \in \Pi(\mu,\nu)} \iint d_{g}(x,y)^r \d \pi (x,y)}^{1/r},\]
where $\Pi(\mu,\nu)$ denotes the set of all Borel probability measures on $\M \times \M$ with marginals $\mu$ and $\nu$. For further details on Wasserstein distances, we refer the reader to \cite{Vil03, Vil09, RKSF, BGL}.
Note that this distance can also be expressed as the expectation of a cost function over a suitable class of random variables
\begin{equation}\label{eq:ebw}
\Wasserstein_{g,r}(\mu, \nu) \;:=\; \parent{\inf_{\substack{ X \sim \mu \\  Y \sim \nu}} \EE\parent{d_{g}( X, Y)^r}}^{1/r}.
\end{equation}

\paragraph{Contraction property}
This variational formulation allows us to derive a contraction property. Let $h: \RR^p \lms \M$ be a Lipschitz map, where $\RR^p$ is equipped with the Euclidean distance $\vabs{.}$, and $\M$ with the Riemannian distance $d_g$. Since
\[\EE\parent{d_{g}\parent{h( X), h( Y)}^r} \leq \norme{h}{\mathrm{Lip}}^r \EE\vabs{ X- Y}^r,\]
where $\norme{h}{\mathrm{Lip}} \= \sup_{x\neq y}\frac{d_{g}(h(x),h(y))}{|x-y|}$, we obtain the following lemma.

\begin{lem}[Contraction property]\label{lem:wcontraction}
Let $m_1$ and $m_2$ be two probability measures on $\RR^p$, and let $h: (\RR^p, d_{\I_n}) \lms (\M, d_g)$ be a Lipschitz map.
Then for all $r\geq1$ \[\Wasserstein_{g,r}(m_1 \circ h^{-1}, m_2 \circ h^{-1}) \leq \norme{h}{\mathrm{Lip}}\Wasserstein_{\I_n,r}(m_1, m_2),\]
where $\Wasserstein_{\I_n,r}$ denotes the Euclidean Wasserstein distance.
\end{lem}

\paragraph{Dual representation}
Recall that the Wasserstein distance admits a dual formulation, which is particularly useful in the study of functional inequalities.
From the dual Monge--Kantorovitch representation (see, e.g., \cite{Vil03, BGL14})
\begin{equation}\label{eq:drwr}
\Wasserstein_{g,r}^r(\mu, \nu) \= \sup_{f \in \L^1(\mu)} \parent{\int Q_r f \d \nu - \int f \d \mu},
\end{equation}
where $Q_r f (x) \= \inf_{y \in \E} \parent{f(y) + d_{g}(x,y)^r}$. As established earlier, Lipschitz functions correspond to functions with bounded Riemannian gradients. In the case $r=1$, the dual representation can be restricted to the class of $1$-Lipschitz functions (see \cite{Vil03}). Combined with \eqref{eq:dgamma}, this yields
\[
\Wasserstein_{g,1}(\mu, \nu) \= \sup_{\text{Lip}(f) \leq 1} \parent{\int f \d(\mu - \nu)} \= \sup_{\norme{\Gamma f}{\infty} \leq 1} \parent{\int f \d(\mu - \nu)}.
\]

\paragraph{Diagonal case}
Consider the diagonal setting induced by the metric in \eqref{eq:mdc}. Combining \eqref{eq:dA} with the expectation-based formulation \eqref{eq:ebw} yields the following more explicit expressions for $r=1$, $r=2$, and more generally for any $r \geq 0$
\begin{eqnarray*}
\Wasserstein_{g,1}(\mu, \nu) &\=& \inf_{( X, Y)} \EE\intervalle{\sqrt{\sum_{i=1}^n \vabs{A( X_i) - A( Y_i)}^2}\,},\\
\Wasserstein_{g,2}(\mu, \nu) &\=& \parent{\inf_{( X, Y)} \sum_{i=1}^n \EE\intervalle{\vabs{A( X_i)-A( Y_i)}^2}}^{1/2},\\
\Wasserstein_{g,r}(\mu, \nu) &\=& \parent{\inf_{( X, Y)} \EE\intervalle{\parent{\sum_{i=1}^n \vabs{A( X_i)-A( Y_i)}^2}^{r/2}}}^{1/r}.
\end{eqnarray*}
An application to the Dyson--Laguerre dynamics is presented in Section \ref{sec:wgdl}.\\

\subsection{Dynamical consequences}

Throughout this section, let $\mu$ and $\nu$ be two probability measures on $\M$, and let $(\P_t)_{t \geq 0}$ be a Markov semigroup associated with the diffusion generator $L$ \eqref{eq:dol}, having invariant measure $\pi$.

\paragraph{Regularization}
In this paragraph, we establish a bound on the relative entropy by the intrinsic Wasserstein distance, in the spirit of \cite[Section 3.4]{CF24}, for diffusion processes under curvature conditions. This property may be interpreted as a Kullback--Leibler regularization: if the initial law is not absolutely continuous with respect to the invariant measure, the relative entropy is infinite at time $t=0$, whereas the Wasserstein distance remains finite whenever the second moment is finite.

The following lemma is a natural consequence of curvature conditions, which are known to imply local log-Sobolev and Poincaré-type inequalities. To establish the regularization property, we adapt the argument from \cite[Lemma 4.2]{BGL} to the Riemannian setting, and use the local Poincaré inequality instead of the local log-Sobolev inequality (see \cite[Section 5]{Ane00} for the terminology). The proof remains valid when using the log-Sobolev inequality.

\begin{lem}[Regularization]\label{lem:reg}
Consider the diagonal setting given by the metric in \eqref{eq:mdc}. If the diffusion satisfies $\C\D(\rho, \infty)$ for some $\rho\in\RR$, then for all $t \> 0$,
\[\Kullback(\mu\P_t \mid \pi) \leq \frac{e^{-2\rho t}}{4 (1 - e^{-2\rho t})} \cdot \Wasserstein_{g,2}^2(\mu, \pi).\]
\end{lem}
\noindent The proof presented here differs slightly from the one in \cite[Lemma 4.2]{BGL}.
\begin{proof}[Proof of Lemma \ref{lem:reg}]
Fix $t>0$. By time reversibility of the semigroup, we begin by rewriting the entropy \( \text{Ent}_\pi (\P_t f) \) as
\[\text{Ent}_\pi (\P_t f) \;:=\; \int (\P_t f) \log (\P_t f) \, \d\pi \= \int f \P_t (\log \P_t f) \, \d\pi.\]
Our next step is to control \( \P_t (\log \P_t f) \).

Let \( x, y \in \RR^n \), and let $\gamma$ be the geodesic joining $x$ to $y$, as given in \eqref{eq:gef}. Let \( v: [0, t] \to [0, t] \) be a smooth speed function such that \( v(0) = 0 \) and \( v(t) = t \). Define $h \= \log$, and for each $0 \leq s \leq t$, set
\[f_s \;:=\; \P_{2t-s}f, \qquad u (s) \;:=\; \gamma\parent{\frac{v(s)}{t}} \and \phi(s, u(s)) \;:=\; \P_s \parent{ h \circ f_s } (u(s)).\]
By Lemma \ref{lem:der}, the Cauchy--Schwarz inequality and using $\vabs{\nabla_g f}_{g} \= \sqrt{\sum_{i=1}^{n} g_{ii} (g^{ii} \partial_i f)^2} \= \sqrt{\Gamma f}$, we obtain
\begin{eqnarray*}
\frac{\d}{\d s} \phi(s, u(s)) &\=& \P_s \parent{(h''\circ f_s) \Gamma(f_s) } \;+\; \ps{\nabla_g \P_s(h\circ f_s)
, u'(s)}_{u(s)}\\
	&\leq& {\P_s \parent{(h''\circ f_s) \Gamma(f_s) }} \;+\; \vabs{\nabla_g \P_s(h\circ f_s)
	}_{u(s)}.\vabs{u'(s)}_{u(s)}\\
	&\leq& {\P_s \parent{(h''\circ f_s) \Gamma(f_s) }} \;+\; \sqrt{\Gamma \P_s(h\circ f_s)
	}.\vabs{u'(s)}_{u(s)}.
\end{eqnarray*}
Recall from \cite{Ane00} that the curvature-dimension condition $\C\D(\rho, \infty)$ entails a sub-commutation inequality between the semigroup $(\P_s)_{s \geq 0}$ and the carré du champ operator $\Gamma$; namely, for all $s \> 0$,
\begin{equation} \label{eq:lpi} \Gamma \P_s f \leq e^{-2\rho s} \P_{s}\Gamma f.\end{equation}
Although the original proof in \cite{BGL} relies on a local log-Sobolev inequality, it is known that in the case of diffusion operators, the curvature condition, sub-commutation, and weak/strong functional inequalities are equivalent. Hence, the required inequality holds. However, we demonstrate that the argument carries through even when the local log-Sobolev inequality is replaced by a local Poincaré inequality.

Using that $\Gamma(h\circ f_s) \= (h'\circ f_s)^2 \Gamma(f_s)$, we deduce from \eqref{eq:lpi} that
\[\vabs{\nabla_g \P_s(h\circ f_s)}_{u(s)} \= \sqrt{\Gamma \P_s(h\circ f_s)} \leq e^{- \rho s} \sqrt{\P_s\Gamma(h\circ f_s)} \leq e^{- \rho s} \sqrt{\P_s\parent{(h'\circ f_s)^2 \Gamma(f_s)}}.\]
Now observe that $u'(s) \;:=\; {\frac{v'(s)}{t}} \gamma'\parent{\frac{v(s)}{t}}$ and since $\gamma$ is a geodesic, we obtain from \eqref{eq:ngpd} that
\[\vabs{u'(s)}_{u(s)} \= \frac{\vabs{v'(s)}}{t} \vabs{\gamma'\parent{\frac{v(s)}{t}}}_{\gamma\parent{\frac{v(s)}{t}}} \= \frac{\vabs{v'(s)}}{t} d_{g}(x,y).\]
Since $h \= \log$, we thus find
\begin{eqnarray*}
\frac{\d}{\d s} \phi(s, u(s)) &\leq& {\P_s \parent{(h''\circ f_s) \Gamma(f_s) }} + {\frac{d_{g}(x,y)}{t} |v'(s)| e^{- \rho s}}\sqrt{\P_s\parent{(h'\circ f_s)^2 \Gamma(f_s)}}\\
	&\leq& - \P_s \parent{\frac{\Gamma(f_s)}{f_s^2}} + 2 \parent{\frac{d_{g}(x,y)}{2t} |v'(s)| e^{- \rho s}} \sqrt{\P_s \parent{\frac{\Gamma(f_s)}{f_s^2}}}.
\end{eqnarray*}
Using the inequality $-a^2 + 2ab \leq b^2$ for $a,b\in \RR$, and taking
\[a \= \sqrt{\P_s \parent{\frac{\Gamma(f_s)}{f_s^2}}}, \qquad b \= \frac{d_{g}(x,y)}{2t} |v'(s)| e^{- \rho s},\]
we are led to the bound
\[\frac{\d\phi}{\d s} \leq  -a^2 + 2ab \leq b^2 \= \frac{d_{g}(x,y)^2}{4t^2} |v'(s)|^2 e^{- 2 \rho s}.\]
Integrating in time, we get
\[\P_t (\log \P_t f)(x) - \log \P_{2t} f(y) \leq \frac{d_{g}(x,y)^2}{4t^2} \int_0^t |v'(s)|^2 e^{- 2 \rho s} \d s.\]
Denoting $S = 4(e^{2 \rho t}-1)$ and choosing the optimal interpolation speed
\[v(s) = t \frac{e^{2 \rho s} -1}{e^{2 \rho t} -1}, \quad 0 \leq s \leq t,\]
we arrive at
\[\P_t (\log \P_t f)(y) \leq \log \P_{2t} f(x) + \frac{1}{S} d_{g}(x,y)^2.\]
For fixed \( y \), define $\phi = h\circ f_0 =\log \P_{2t} f$. Taking the infimum over \( x \)
\[\P_t (\log \P_t f)(y) \leq \inf_{x} \parent{\phi(x) + \frac{1}{S} d_{g}(x,y)^2} \;=:\; Q_{2,S} \phi(y).\]
However, by Jensen’s inequality
\[\int \phi \, \d\pi = \int \log \P_{2t} f \, \d\pi \leq \log \left(\int \P_{2t} f \, \d\pi \right) = 0,\]
and thus,
\[\P_t (\log \P_t f) \leq Q_{2,S} \phi- \int \phi \, \d\pi.\]
Therefore, taking the supremum over all bounded measurable functions $f$, we deduce from the dual Monge–Kantorovich formulation \eqref{eq:drwr} that
\[\text{Ent}_\pi (\P_t f) = \int f \P_t (\log \P_t f) \, \d\pi \leq \sup_\varphi \left[ \int Q_{2,S} \varphi \, \d\mu - \int \varphi \, \d\pi \right] \= \frac{\Wasserstein_{g,2}(\mu,\pi)^2}{4(e^{2 \rho t}-1)},\]
which concludes the proof.
\end{proof}

The following lemma is a consequence of standard semigroup properties combined with the chain rule. See also \cite{Mon18}.

\begin{lem}\label{lem:der}
Let $t>0$ be fixed, and define for any $s\in[0,t]$ and function $f$,
\[f_s \;:=\; \P_{2t-s}f, \quad \phi(s,y) \;:=\; \P_s \parent{ h \circ f_s } (y) \quad \mtext{and} \quad h \= \log.\]
Then along any smooth curve $\gamma : [0,t] \lra \M$, one has
\[\frac{\d}{\d s} \phi(s, \gamma(s)) \= \P_s \parent{(h''\circ f_s) \Gamma(f_s) } \;+\; \ps{\nabla_g \P_s(h\circ f_s) (\gamma(s)), \gamma'(s)}_{\gamma(s)}.\]
\end{lem}

\begin{proof}[Proof of Lemma \ref{lem:der}]
It suffices to prove that 
\[\frac{\d}{\d s} \P_s (h \circ f_s) \= \P_s \parent{(h''\circ f_s) \Gamma(f_s)}.\]
Using the kernel representation of the semigroup $\P_s (h \circ f_s) = \int (h\circ f_s) \, p_s \d \pi$, differentiating in time and applying the semigroup derivative yields
\begin{eqnarray*}
\frac{\d}{\d s} \P_s (h \circ f_s) &\=& (\partial_s \P_s)(h\circ f_s) \;+\; \P_s\parent{\partial_s (h\circ f_s)}\\
	&\=& (\P_s \G)(h\circ f_s) \;+\; \P_s\parent{(h'\circ f_s) (\partial_s f_s)}\\
	&\=& \P_s \parent{\G(h\circ f_s) \;+\; (h'\circ f_s) (\partial_s f_s)}.
\end{eqnarray*}
Next, observe that \[\partial_s f_s \= \partial_s (\P_{2t-s} f) \= (-1) (\partial_s \P)_{2t-s} f \= - \G \P_{2t-s} f \= -\G f_s.\]
Therefore,
\[\frac{\d}{\d s} \P_s (h \circ f_s) \= \P_s \parent{\G(h\circ f_s) - (h'\circ f_s) (\G f_s)}.\]
Recall that $\G \= (g \Delta - \nabla U \cdot \nabla)$. Applying the chain rule yields \[\nabla (h\circ f_s) \= (h'\circ f_s) \nabla f_s, \qquad \partial_{ii}^2 (h\circ f_s) \= (h''\circ f_s) (\partial_i f_s)^2 + (h'\circ f_s) \partial_{ii}^2 f_s,\] which imply \[\G(h\circ f_s) \= (h''\circ f_s) \Gamma(f_s) + (h'\circ f_s) \G f_s.\]
\end{proof}

\paragraph{Exponential decay}
The following result corresponds to the exponential decay of the intrinsic Wasserstein distance associated with a curved diffusion process. We follow the proof given in \cite{vRS05}, and reproduce it here for the reader's convenience. Note that the converse implication is also established in the same reference when considering Lipschitz functions with respect to the intrinsic Riemannian distance; however, we do not address that direction. The proof relies on the parallel displacement coupling, also known as the Kendall--Cranston coupling, originally introduced in \cite{Ken86a, Cra91} (see also \cite{Wan97b}).

\begin{lem}[Exponential decay]\label{lem:expdec}
For any diffusion satisfying $\C\D(\rho, \infty)$ with $\rho \in \RR$, the following inequality holds for all $t > 0$
\[\Wasserstein_{g,2}(\mu\P_t, \nu\P_t) \leq e^{-\rho t} \cdot \Wasserstein_{g,2}(\mu, \nu).\]
\end{lem}

\noindent To prove this property, we employ tools from stochastic differential geometry, using a coupling process via parallel transport; see Section \ref{sec:sdg}. Note that this step can alternatively be shown by a central limit theorem for coupled geodesic random walks.

\begin{proof}
From the Kendall--Cranston coupling (see Section \ref{sec:KCc}), we obtain the following for $X_t$ and $Y_t$ started at two distinct points $x,y \in \M$.
Since the distance process $(d_g(X_t, Y_t))_{t\geq0}$ is bounded above by an Ornstein--Uhlenbeck process in \eqref{eq:dpbaou}, whose expectation at time $t$ is $e^{-\rho t} d_g(X_0, Y_0) = e^{-\rho t} d_g(x, y)$, we obtain
\begin{equation}\label{eq:eed}
\fa t\geq 0, \quad \EE_{(x,y)}\parent{d_g(X_t, Y_t)^r}^{1/r} \leq e^{-\rho t} d_g(x,y).
\end{equation}
Suppose now that $X_0 \sim \mu$, $Y_0 \sim \nu$.
Let $\lambda$ be an optimal coupling of $\mu$ and $\nu$ with respect to $\Wasserstein_{g,r}$, i.e.,
\[\Wasserstein_{g,r}^r(\mu,\nu) \= \int d_g(x,y)^r \lambda(\d x, \d y).\]
Let $\lambda_t \;:=\; \lambda \Q_{2t}$ be a coupling of $\mu\P_t$ and $\nu \P_t$. Hence, using the properties of the Markov semigroup and the equation \eqref{eq:eed}, we get
\begin{eqnarray*}
\Wasserstein_{g,r}^r(\mu\P_t, \nu\P_t) \;\leq\; \int d_g(u,v)^r \lambda_t(\d u, \d v) &=& \iint d_g(u,v)^r \Q_{2t}^{(x,y)}(\d u, \d v) \lambda(\d x, \d y)\\
	&=& \int \EE_{(x,y)} \parent{d_g( X_{2t},  Y_{2t})^r} \lambda(\d x, \d y)\\
	&\leq& e^{- 2 r \rho t} \int d_g(x,y)^r \lambda(\d x, \d y)\\
	&=& e^{- 2 r \rho t} \Wasserstein_{g,r}^r(\mu,\nu).
\end{eqnarray*}
\end{proof}

\subsection{Stochastic differential geometry}\label{sec:sdg}

Recall that \((\M,g)\) is a complete $d$-dimensional Riemannian manifold, so that each tangent fibre $T_x \M$ is Euclidean and is isometric to $\RR^d$.
We consider the second-order diffusion operator defined by \eqref{eq:dol}, equivalently $L = \Delta_g + Z$, where $\Delta_g$ is the Laplace--Beltrami operator and $Z$ a $C^1$-vector field (see \cite[Section C.5]{BGL} and \cite{Wan97b}).

The goal of this section is to construct a stochastic process $(X_t)_{t\geq0}$ on $\M$ whose generator is exactly $L$; see \cite{Ken87}, \cite[Section 2]{Ken86b}, or \cite[Chapter 2]{Hsu02}.

\paragraph{Stratonovich differential}
A basic question is whether one can define the stochastic differential so that a change of coordinates does \emph{not} alter the generator $L$. With Itô calculus this generally fails: if $Y_t = \psi(X_t)$, the Itô correction (a quadratic-variation drift) appears, so the generator of $Y$ is not simply the pushforward of the generator of $X$. In particular, a Brownian motion defined in one chart need not remain Brownian after a nonlinear change of coordinates; Itô calculus is thus not intrinsic on manifolds.
The solution is to adopt a notion of stochastic integration compatible with the geometry: the \emph{Stratonovich} formulation (denoted $\circ \d$). It obeys the usual chain rule and respects smooth coordinate changes. Concretely, for $f\in C^{\infty}(\M,\RR)$,
\begin{equation}\label{eq:dpc}
\d f(X_t) \= (T f)_{X_t} \circ \d X_t,
\end{equation}
where $T f$ is the tangent map (differential/pushforward) of $f$. In this form, Stratonovich SDEs transform by pushforward under diffeomorphisms, so the intrinsic expression of the generator is preserved and we may work directly on $\M$ without tracking charts.

\paragraph{Orthonormal frame bundle}
Formula \eqref{eq:dpc} indicates that $\d X$ should be viewed as a tangent vector at $X_t$, i.e., an element of $T_{X_t}\M$, since the Stratonovich differential behaves as an ordinary differential. 
Recall that the only datum at our disposal is $(X_t)_{t\geq0}$ specified through its generator $L$.
Our stochastic calculus lives in $\RR^d$, while each tangent space $T_{x} \M$ is isometric to $\RR^d$. Thus,
it is natural to choose an $\RR^d$-valued semimartingale $(M_t)_{t\geq0}$ and, for each $X_t \in \M$, an isometry $U_t:\RR^d \lra T_{X_t} \M$ so that
\begin{equation} \label{eq:target}
\d X_t \= U_t \circ \d M_t.
\end{equation}
The process $(U_t)_{t\geq0}$ then takes values in the \textit{orthonormal frame bundle}
\[O(\M) := \bigsqcup_{x\in\M} O_x(\M) \quad \mtext{with} \quad O_x(\M) := \event{u_x : \RR^d \lra T_x \M \textm{linear isometry}}.\]
%l'union de toutes les fibres $O_x(\M)$, ensemble de toutes les isométries de $\RR^d$ vers $T_x \M$, lorsque $x$ parcourt $\M$.
This bundle carries a natural smooth structure, so it makes sense to treat $(U_t)_{t\geq0}$ as a semimartingale on $O(\M)$.
Informally, $\d X_t$ can be ``pulled back'' through $U_t$ to a Stratonovich differential on $\RR^d$, namely $\d M_t$. In schematic form
\begin{equation*}\label{eq:msvr} M_t \textm{on } \RR^d \xra{\text{moving frame}}{} U_t \textm{on } O(\M) \xra{\text{projection}}{p} X_t \textm{on } \M. \end{equation*}

\paragraph{From $O(\M)$ to $\M$: moving frame}
Following \cite[Section 3.1]{Ken87}, as a first step we explain the dynamics of the process $(U_t)_{t\geq0}$, which is the \textit{moving (orthonormal) frame} attached to the path $X_t$.

\noindent There is a natural projection $p: O(\M) \lra \M$ sending a frame $u_x \in O_x(\M)$ to its base point $x\in\M$. Since $U_t$ is a frame at $X_t$, we must have
\begin{equation}\label{eq:piso}
p(U_t) = X_t.
\end{equation}

\noindent Because $(U_t)_{t\geq0}$ is a semimartingale on $O(\M)$, its differential lives in $T_{U_t} O(\M)$.
The only ingredient we currently have is the differential $(\d X_t)_{t\geq0}$, which takes values in the tangent bundle $T\M$.
Thus, for each isometry $u\in O(\M)$, we must choose a linear map from $T_{p(u)} \M$ to $T_u O(\M)$ that preserves orientation. A moving frame $(U_t)_{t\geq0}$ has two possible types of motion. A \emph{vertical} motion, which is an in-place rotation of the frame (the base point does not move). A \emph{horizontal} motion, which moves the base point along the trajectory by following the connection while preserving its orientation in a geometrically ``natural'' way, so that the axes remain aligned with those of the previous position, i.e., \textit{parallel transport}.
The second motion is determined by the unique natural connection (the Levi-Civita connection) that is torsion-free, metric-compatible, and angle-preserving, and induces the \textit{horizontal lift}, which formally corresponds to the map
\[\H : T \M \lra T O(\M),\]
where for each frame $u_x\in O_x(\M)$, $\H_u : T_x \M \lra T_u O(\M)$. Differentiating \eqref{eq:piso} gives equivalently $(T p) \circ (H_u) = \mathrm{Id}_{T_{p(u)} \M}$ for each $u \in O(\M)$. Thus, we obtain
\[\d U_t = H_{U_t} \circ \d X_t.
\]

\paragraph{From $\RR^d$ to $\M$: diffusion on a manifold}
As a second step we define a semimartingale $(M_t)_{t\geq0}$ in $\RR^d$ so that its stochastic development yields the diffusion $(X_t)_{t\geq0}$ on $\M$ with generator $L$. We set $(M_t)_{t\geq0}$ to satisfy
\[\d M_t \= \sqrt{2} \d B_t + A_t \d t.\]
On the one hand, given the parallel transport $(U_t)_{t\geq0}$, the target relation \eqref{eq:target} yields
\[\d M_t = U_t^{-1}\, \circ \d X_t,\]
where $(U_t^{-1})$ is itself a semimartingale on the appropriate bundle, so the stochastic differential above is well defined.

\noindent On the other hand, by applying the Stratonovich chain rule to $f\in C^\infty(\M,\RR)$ we obtain
\[\d f (X_t) \= (T f)_{X_t} \circ \d X_t \= (\d f)_{X_t} (U_t \circ \d M_t) \= \sum_{i=1}^d (\d f)_{X_t} (U_t . e^i) \circ \d M_t^i,\]
for any fixed Euclidean basis $(e^i)$ of $\RR^d$. As shown in \cite[Section 1]{Cra91}, a standard computation yields
\[\d f (X_t) = \sqrt{2} \sum_{i=1}^d (U_t e^i f) (X_t) \d B_t^i + \parent{\sum_{i=1}^d U_t e^i A_t^i  f (X_t) + \Delta_g f(X_t)} \d t.\]
To ensure that the drift equals $Lf(X_t)$, choose $A_t$ so that its $e^i$-component satisfies $Z^i = U_t e^i A_t^i$; equivalently
\[A_t := U_t^{-1} Z.\]
Consequently,
\[\d M_t \= \sqrt{2} \, \d  B_t + U_t^{-1} Z (X_t) \d t.\]

In summary, the analysis-synthesis approach leads to the following standard frame-bundle formulation of diffusions on $\M$ for $X_0 = x$,
\begin{equation}\label{eq:sdg}
\begin{acc}
$\dps \d M_t \= \sqrt{2} \, \d  B_t + U_t^{-1} Z (X_t) \d t$,\\
$\dps \d U_t \= H_{U_t} U_t \circ \d M_t$,\\
$\dps X_t \= p(U_t)$.
\end{acc}
\end{equation}

\subsection{Coupling by parallel displacement (Kendall--Cranston coupling)}\label{sec:KCc}

The basic idea of coupling methods is to realize two random processes on a common probability space so that each component has the prescribed marginal law.
Working on a shared space exposes the joint distribution and allows the components to interact through common events (e.g. ``$X=Y$''), thereby revealing their relative behavior.

Let \( (X_t) \) and \( (Y_t) \) be Markov processes on a state space $\M$ with transition semigroup $(\P_t)_{t\geq0}$ and initial laws $X_0 \sim \mu$, $Y_0 \sim \nu$. A \emph{coupling} is a Markov process $(\widetilde{X}_t, \widetilde{Y}_t)_{t\ge0}$ on the product space \( (\M \times \M, \PPP \times \PPP) \) such that, for each $t$, the marginals are $\mu \P_t$ and $\nu \P_t$.
In particular, \( (X_t) \) and \( (Y_t) \) need not be defined on the same probability space, but \( (\widetilde{X}_t) \) and \( (\widetilde{Y}_t) \) must be.

There are several standard couplings (e.g. independent, synchronous couplings), tailored to different objectives.
A classical goal is to arrange for the two processes to meet and then evolve together--this is a \emph{coalescent coupling} (also called \textit{successful}), meaning that the paths coincide from some time onward.
The quantity encoding this is the \emph{coupling time}, defined by the random variable
\[T \;:=\; \inf \event{t > 0 \;:\; X_t \= Y_t}.\]
This time can be infinite in certain settings (e.g., the synchronous coupling of two Brownian motions on $\RR^d$ with $d\geq2$). In contrast, reflection-type constructions can make $T$ finite under suitable conditions (see e.g., \cite{LR, Ebe16}).\\

Reflection coupling is designed to produce a coupling in which the two processes eventually meet, i.e., a \emph{coalescent} (successful) coupling. The guiding idea is to drive the second process so that, up to the meeting time, its trajectory is the mirror image of the first. However, specifying the reflected construction alone does not, by itself, ensure coalescence in all settings; additional structural assumptions on the dynamics or the geometry are typically required.

Intuitively, in $\RR^d$ the relevant geometry as the two processes approach each other is the line through $X_t$ and $Y_t$. We wish to steer $Y_t$ toward $X_t$ along this line, that is, along the direction $-v_t$, where $v_t := \overrightarrow{X_tY_t}/ \norme{Y_t - X_t}{}$ is the unit vector from $X_t$ to $Y_t$. Since changing the drift or the diffusion matrix would alter the generator, the only degree of freedom is the driving Brownian motion. To construct the ``mirrored'' noise $(\widetilde B_t)_{t\ge0}$, fix an orthonormal basis $(e_1:=v_t, e_2,\dots,e_d)$ and flip the $e_1$--component of the noise:
\[\d \widetilde{B}_t \;:=\; \sum_{i=2}^d \ps{B_t, e_i} e _i - \ps{B_t, e_1} e_1 \= \d B_t - 2 \ps{B_t, e_1} e_1 \= (I_d - 2 e_1 e_1^\top) \d B_t.\]
This way of reversing the component along the preferred direction $-e_1$ is straightforward in Euclidean space and extends to Riemannian manifolds via the \emph{mirror map}; in $\RR^d$ is simply $m_{X_t,Y_t}: z \lms (I_d - 2 e_1 e_1^\top)\, z$).

Even in $\RR^d$, an independent coupling of two Brownian motions need not yield coalescence.
In dimension $d=1$, coalescence does occur: the difference $(B_t-\widetilde B_t)_{t\ge0}$ is a one-dimensional Brownian motion, which is recurrent and hence hits every point, in particular $0$, almost surely.
For $d=2$, coalescence need not occur: although the difference is a two-dimensional Brownian motion, planar Brownian motion is recurrent while singletons are polar, so it almost surely does not hit a fixed point (in particular, $0$).
For $d \geq 3$, Brownian motion is transient, so the question does not even arise.

Because the diffusion has a non-constant diffusion matrix, the classical Euclidean reflection coupling does not apply. We therefore need a coupling adapted to the manifold's geometry. The coupling by parallel displacement, also known as Kendall--Cranston coupling, provides a suitable construction; see \cite{Ken86a, Cra91} (or \cite{Wan96, CW94}).

The Euclidean construction relies on the uniqueness of the minimizing path from $X_t$ to $Y_t$. If two distinct minimizers existed, which ``mirror direction'' would be chosen? In Euclidean space this issue never arises, but on a Riemannian manifold it can occur (for example between antipodal points on a sphere). We therefore proceed in two steps: (i) on the region where the minimizing geodesic is unique, we define the mirror map that yields the symmetric Brownian motion; (ii) when several minimizing geodesics may exist, that is, on the cut locus, we adapt the construction to handle this non-uniqueness.

\paragraph{Mirror map: case of a unique minimizing geodesic} Fix a time $t$ at which there is a unique minimizing geodesic $\gamma_t$ from $X_t$ to $Y_t$.
As in the Euclidean setting, the reflection coupling is determined by two ingredients: the Riemannian Brownian noise (the only degree of freedom), and the unit tangent vector \(v_t\).
The \emph{Riemannian} Brownian motion associated with $(X_t)_{t\geq0}$ defined by \eqref{eq:sdg} is
\[\d B_t^g := U_t \circ \d B_t \in T_{X_t }\M,\]
where \((B_t)_{t\ge 0}\) is a Brownian motion in \(\RR^d\) (see \cite[Section 2]{Ken87}).
The \emph{unit tangent vector} associated with \((X_t)_{t\ge 0}\) and \((Y_t)_{t\ge 0}\) is
\[v_t := \dot{\gamma}_t(0) \in T_{X_t} \M.\]

The mirror map is constructed in two steps.
First, transport the unit tangent vector $v_t \in T_{X_t} \M$ along the minimizing geodesic $\gamma_t$ in a geometrically consistent way (parallel transport) to obtain the unit vector $v_t' \in T_{Y_t} \M$.
Second, ``reverse'' the component of the Riemannian Brownian motion in that direction, by reflecting across the hyperplane $\H \subset T_{Y_t} \M$ orthogonal to $v_t'$. The resulting unit vector is denoted $m_{X_t, Y_t} (v_t) \in T_{Y_t} \M$ and defines the mirror map.
Note that the mirror map $m_{x,y}$ is an isometry from $T_x \M$ onto $T_y \M$.
Schematically:
\begin{figure}[h!]
\centering
\includegraphics{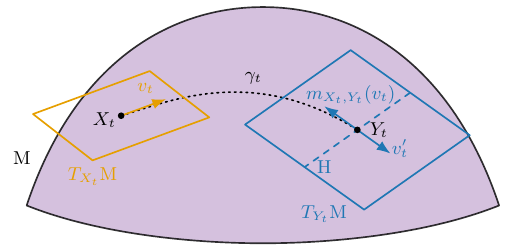}
\caption{Mirror map.}
\label{Mirror map}
\end{figure}

Thus the Riemannian Brownian noise for \(Y_t\) is defined by $\d \widetilde{B}_t^g := m_{X_t, Y_t} \circ \d B_t^g$.
More generally, linking this to \(\RR^d\) martingales, we define the Brownian motion $(\widetilde{B}_t)_{t\geq0}$ in \(\RR^d\) from which the semimartingale $(N_t)_{t\geq0}$ driving \(Y_t\) arises via
\[\d B_t \in \RR^d \xra{U_t}{} T_{X_t} \M \xra{m_{X_t,Y_t}}{} T_{Y_t} \M \xra{V_t^{-1}}{} \RR^d \ni \d \widetilde{B}_t.\]
In particular, since $V_t^{-1} m_{X_t,Y_t} U_t$ is an isometry of $\RR^d$, we have that $(\widetilde{B}_t)_{t\geq0}$ is a standard Brownian motion in \(\RR^d\).

Given \((X_t)_{t\ge0}\) solving \eqref{eq:sdg} from \(x\), we couple it with \((Y_t)_{t\ge0}\) that satisfies the same dynamics but is driven by the reflected Brownian noise \((\widetilde{B}_t)_{t\geq 0}\).
Since it is built from the differential of the exponential map, on the domain where this is a diffeomorphism, it follows that the dependence of $m_{x,y}$ on $x$ and $y$ is smooth whenever $x\neq y$ and we are away from the cut locus (see \cite{Ken86a}).
This system admits a unique solution up to the first entrance into the cut locus or until coalescence.

\paragraph{Cut locus: case of multiple minimizing geodesics}
On some Riemannian manifolds, fixing a starting point $x$, there exist endpoints that can be reached by more than one minimizing geodesic. For instance, on the sphere, the antipode of $x$ is the canonical example. The set of all such points is the \textit{cut locus} of $x$, denoted $C(x)$.

Since the reflection coupling is built entirely from minimizing geodesics, the construction breaks down when $Y_t$ reaches the cut locus of $X_t$. At that time the minimizing geodesic is not unique, the reflection is no longer well defined, since multiple mirror directions are possible. In addition, the distance function $d_g$ loses regularity on the cut locus.

\paragraph{Kendall--Cranston coupling}
A coupling of solutions to the stochastic differential equation on a manifold $\M$ defined by \eqref{eq:sdg} is constructed as follows. Fix initial data $(x,y)$ with $x \neq y \in \M$ and $y \notin C(x)$.

\noindent The process \((X_t)_{t\ge0}\) is the solution to \eqref{eq:sdg} started at \(x\).

\noindent The process \((Y_t)_{t\ge0}\) started at \(y\) is defined by the following alternating scheme:
$(i)$ for all $t<T$ with $Y_t \notin C(X_t)$, $Y_t$ evolves with the same dynamics as $(X_t)_{t \geq 0}$ but driven by the reflected Brownian noise \((\widetilde{B}_t)_{t\geq 0}\)
\[
\begin{acc}
$\dps \d N_t \= \sqrt{2} \d \widetilde{B}_t + V_t^{-1} Z ( Y_t) \d t, \quad \d \widetilde{B}_t \= V_t^{-1} m_{X_t, Y_t} U_t \, \d  B_t$,\\
$\dps \d V_t \= H_{V_t} \circ \d Y_t$,\\
$\dps  Y_t \= p(V_t), \quad  Y_0 \= p(V_0)$,
\end{acc}
\]
$(ii)$ while \(Y_t\in C(X_t)\), \(Y_t\) evolves with the same dynamics as $(X_t)_{t \geq 0}$ but driven by an independent Brownian motion, until the first exit time from the cut locus,
$(iii)$ repeat steps $(i)$ and $(ii)$ alternately until the coupling time $T$, and for $t \geq T$, set $Y_t = X_t$.

\paragraph{Coalescence of the coupling}
Recall that coalescence is characterized by the almost sure finiteness of the coupling time \(T\) defined above. A natural way to study \(T\) is via the \emph{distance process} \(r_t := d_g(X_t,Y_t)\), since
\[T := \inf\event{t>0, \; r_t = d_g(X_t, Y_t) = 0}.\]
Accordingly, we study the decay of \((r_t)_{t\ge0}\) toward a possible hit of \(0\). The approach is to dominate \(r_t\) by a suitable continuous martingale that hits \(0\) almost surely when the Ricci curvature of \(\M\) is nonnegative. This upper bound follows from a stochastic analysis of the one-dimensional nonnegative process \((r_t)_{t\ge0}\).

First, the following result quantifies whether the two processes move closer.
Itô calculus applied to the distance process introduces a local time of $(Y_t)_{t \geq 0}$ on the cut locus of $(X_t)_{t \geq 0}$. Away from the cut locus, the process $(X_t, Y_t)_{t\geq 0}$ is a diffusion on $\M\times\M$ whose generator is derived from the first and the second variation of arc length; see \cite[Section 1]{Cra91}, \cite[Lemma 2.1]{CW94}, or \cite{Wan96}). Hence, before time $T$
\[\d r_t \= 2 \sqrt{2} \d b_t + \parent{\int_{0}^{r_t} \sum_{i=2}^d |\nabla_{v_t} J^i |^2 - \ps{R(J^i, v_t) v_t, J^i} + \ps{\nabla_{v_t} Z, v_t} } \d t - \d L_t,\]
where $b_t$ is a one-dimensional Brownian motion,
$(L_t)_{t\geq0}$ is an increasing process with support contained in $\event{t\geq0 : Y_t \in C(X_t)}$, $\d t$ taken to be $0$ when $Y_t \in C(X_t)$ so that when $Y_t \notin C(X_t)$, $\gamma_t : [0,r_t] \lra \M$ represents the unique shortest geodesic between $X_t$ and $Y_t$ with unit tangent vector $v_t$ and $(J^i)_{2\leq i \leq d}$ are Jacobi fields along $\gamma_t$.

Second, we use Riemannian curvature properties
\[\int_{0}^{r_t} \sum_{i=2}^d |\nabla_{v_t} J^i |^2 - \ps{R(J^i, v_t) v_t, J^i} \leq - \int_0^{r_t} \Ric(v_t,v_t).\]
Moreover, from \cite[Section C.6]{BGL14}, the curvature-dimension condition $\C\D(\rho, \infty)$ holds for the generator $L$ in \eqref{eq:dol} if and only if $\Ric - \nabla_U Z \geq \rho g$.
Hence, since $v_t$ is a unit vector (i.e., $\ps{v_t,v_t}_g = \vabs{v_t}_{g}^2 = 1$), we obtain
\begin{equation}\label{eq:dpbaou}
\d r_t \leq 2 \sqrt{2} d b_t - \rho r_t \d t.
\end{equation}
Consequently, the distance process is bounded above by an Ornstein--Uhlenbeck process with initial condition $d_g(x,y) > 0$ (consider the process $\d Z_t = 2 \sqrt{2} \d b_t - \rho Z_t \d t$, and apply Grönwall's lemma to $Z_t - r_t$).
When $\rho \geq 0$, the one-dimensional Ornstein--Uhlenbeck process is recurrent, so it hits $0$ in finite time with probability one.
Since $r_t \geq 0$, it follows that the distance process hits $0$, i.e., $X_t = Y_t$, in finite time almost surely.
Therefore, the coupling is coalescent under nonnegative curvature.

\end{appendix}

\bibliography{bibliographie_laguerre}
\bibliographystyle{abbrv}

\end{document}